\documentclass[11pt,reqno]{amsart}
\usepackage[paperheight=279mm,paperwidth=18cm,textheight=26cm,textwidth=14cm,includehead]{geometry}
\usepackage[utf8]{inputenc}
\usepackage[unicode]{hyperref}
\hypersetup{
bookmarks=true,
colorlinks=true,
citecolor=[rgb]{0,0,0.5},
linkcolor=[rgb]{0,0,0.5},
urlcolor=[rgb]{0,0,0.75},
pdfpagemode=UseNone,
pdfstartview=FitH,
pdfdisplaydoctitle=true,
pdflang=en-US
}

\usepackage{amssymb}
\usepackage{amsmath}
\usepackage{amsthm}
\usepackage{amsfonts}
\usepackage{mathtools}
\usepackage{graphicx}
\usepackage{xcolor}

\numberwithin{equation}{section}
\newtheorem{theorem}{Theorem}[section]
\newtheorem{proposition}[theorem]{Proposition}
\newtheorem{corollary}[theorem]{Corollary}
\newtheorem{lemma}[theorem]{Lemma}

\theoremstyle{definition}

\newtheorem{definition}[theorem]{Definition}
\newtheorem{remark}[theorem]{Remark}
\newtheorem{example}[theorem]{Example}

\DeclareMathOperator{\supp}{supp}

\DeclarePairedDelimiter\abs{\lvert}{\rvert}
\DeclarePairedDelimiter\norm{\lVert}{\rVert}

\DeclarePairedDelimiterX\innerp[2]{\langle}{\rangle}{#1,#2}

\DeclarePairedDelimiterX\Set[1]\{\}{%

#1
}

\newcommand{\La}{\langle}
\newcommand{\Ra}{\rangle}

\newcommand{\pd}{\partial}

\def\supp{\operatorname{supp}}

\newcommand{\ch}{\operatorname{ch}}

\newcommand{\al}{\alpha}

\newcommand{\la}{\lambda}
\newcommand{\vf}{\varphi}
\newcommand{\om}{\omega}
\newcommand{\Om}{\Omega}

\newcommand{\cA}{\mathcal{A}}

\newcommand{\cD}{\mathcal D}
\newcommand{\cE}{\mathcal E}

\newcommand{\cM}{\mathcal M}

\newcommand{\cP}{\mathcal P}

\newcommand{\cR}{\mathcal R}

\newcommand{\bD}{\mathbb D}

\newcommand{\bI}{\mathbb I}

\newcommand{\bR}{\mathbb R}

\newcommand{\bV}{\mathbb V}
\newcommand{\bZ}{\mathbb Z}

\newcommand{\1}{{\mathbf 1}}
\newcommand{\sgn}{\operatorname{sgn}}

\newcommand{\one}{\mathbf{1}}
\newcommand{\dif}{\mathop{}\!\mathrm{d}} 

\def\PZdefchar#1{
  \expandafter\def\csname frak#1\endcsname{\mathfrak{#1}}
  \expandafter\def\csname bb#1\endcsname{\mathbb{#1}}
  \expandafter\def\csname bf#1\endcsname{\mathbf{#1}}
  \expandafter\def\csname scr#1\endcsname{\mathcal{#1}}
  \expandafter\def\csname cal#1\endcsname{\mathcal{#1}}}
\def\PZdefloop#1{\ifx#1\PZdefloop\else\PZdefchar#1\expandafter\PZdefloop\fi}
\PZdefloop abcdefghijklmnopqrstuvwxyzABCDEFGHIJKLMNOPQRSTUVWXYZ\PZdefloop

\hypersetup{
pdftitle={Bi-parameter embedding on bi-tree and bi-disc},
pdfauthor={Arcozzi, Mozolyako, Psaromiligkos, Volberg, Zorin-Kranich},
}

\usepackage{tikz}
\usetikzlibrary{patterns}

\begin{document}

\title[Bi-parameter embedding on bi-tree and bi-disc]%
{Bi-parameter embedding on bi-tree and bi-disc and box condition }
\author[N.~Arcozzi]{Nicola Arcozzi}
\address[N.~Arcozzi]{Universit\`{a} di Bologna, Department of Mathematics, Piazza di Porta S. Donato, 40126 Bologna (BO)}
\email{nicola.arcozzi@unibo.it}
\thanks{Theorem 3.1 was obtained in the frameworks of the project 17-11-01064 by the Russian Science Foundation}
\thanks{NA is partially supported by the grants INDAM-GNAMPA 10017 "Operatori e disuguaglianze integrali in spazi con simmetrie" and PRIN 10018 "Variet\`{a} reali e complesse: geometria, topologia e analisi armonica"}
\author[P.~Mozolyako]{Pavel Mozolyako}
\thanks{PM is supported by the Russian Science Foundation grant 17-11-01064}
\address[P.~Mozolyako]{Universit\`{a} di Bologna, Department of Mathematics, Piazza di Porta S. Donato, 40126 Bologna (BO)}
\email{pavel.mozolyako@unibo.it}
\author[G.~Psaromiligkos]{Georgios Psaromiligkos}
\address[G.~Psaromiligkos]{Department of Mathematics, Michigan Sate University, East Lansing, MI. 48823}
\email{psaromil@math.msu.edu}
\author[A.~Volberg]{Alexander Volberg}
\thanks{AV is partially supported by the NSF grant DMS-160065 and DMS 1900268 and by Alexander von Humboldt foundation}
\address[A.~Volberg]{Department of Mathematics, Michigan Sate University, East Lansing, MI. 48823}
\email{volberg@math.msu.edu}
\author[P.~Zorin-Kranich]{Pavel Zorin-Kranich}
\thanks{PZ was partially supported by the Hausdorff Center for Mathematics (DFG EXC 2047)}
\address[P.~Zorin-Kranich]{Mathematical Institute, University of Bonn, Bonn, Germany}
\email{pzorin@uni-bonn.de}
\begin{abstract}
Coifman--Meyer multipliers represent a very important class of bi-linear singular operators, which were extensively studied and generalized. They have a natural multi-parameter counterpart. Decomposition of those operators into paraproducts, and, more generally to multi-parameter paraproducts is a staple of the theory. In this paper we consider weighted estimates for bi-parameter paraproducts that appear from such multipliers.  Then we apply our harmonic analysis results to several complex variables. Namely, we show that a  (weighted) Carleson embedding from the bi-torus to the bi-disc is equivalent to a simple ``box'' condition, for product weights on the bi-disc and arbitrary weights on the bi-torus.
This gives a new simple necessary and sufficient condition for the embedding of the whole scale of weighted Dirichlet spaces of holomorphic functions on the bi-disc.
This scale includes the classical Dirichlet space on the bi-disc. 
Our result is in contrast to the classical situation on the bi-disc considered by Chang and Fefferman, when a counterexample due to Carleson shows that the ``box'' condition does not suffice for the embedding to hold.
Our result can be viewed as a new and unexpected combinatorial property of all positive finite planar measures.

\end{abstract}
\maketitle



%
%
\maketitle

\section{Coifman--Meyer multipliers and Multi-Parameter paraproducts}
	\label{para}
	
	Let us recall Coifman--Meyer multipliers operators and why  paraproducts appear naturally in their study.
	Let us consider (multi)linear  operators of the  following type:
	\begin{equation*}
	T_m(f, g)\!\! = \!\! \int_{\bR^{(n-1)k}}\!\! \!\! \!\! \!\!  m(\xi) \hat f(\xi_1) \hat f(\xi_2)\cdot\dots \hat f_{n-1}(\xi_{n-1}) e^{i x\cdot (\xi_1+\xi_2+\dots + \xi_{n-1})} d\xi_1\dots d \xi_{n-1},
	\end{equation*}
	where $f, g$ are test functions on $\bR^k$, and now everything depends on the behavior of $m(\xi)$, with  
	$\xi=(\xi_1, \xi_2, \dots, \xi_{n-1})$.
	If for multi-indices $j$ we have
	\begin{equation}
	\label{CM1}
	|\pd^j m(\xi)|\lesssim |\xi|^{-|j|},
	\end{equation}
	then this bilinear operator is called Coifman--Meyer multiplier.

	\bigskip
	
	{\bf Examples.}  1) The Hilbert transform is a Coifman--Meyer multiplier operator, as it can be obtained by letting above $k=1$, $n=2$,  and $m(\xi)= m(\xi_1) = \sgn\, \xi_1$.

	2) Let $k=1$, $n=3$,  $m(\xi)= \sgn\, (\xi_1-\xi_2)$. Then one can easily see the famous bilinear Hilbert transform:
	$$
	B(f, g) = \int f(x+t) g(x-t)\frac{dt}{t},
	$$
	which is not a classical, but generalized Coifman--Meyer multiplier, as the singularity of the multiplier $m$ is not concentrated at the origin as it should be by \eqref{CM1}. Here the singularity is the whole line. (However, as the dimension of the line is obviously smaller than $n/2=3/2$, results of \cite{MuTaTh} show the  range of boundedness of this operator.)
	
	\bigskip
	
	We will consider only operators of Coifman--Meyer type, and the next example can be given by a very important class of operators called paraproducts. For the sake of simplicity our test functions will be only on $\bR^k, k=1,$ but everything has an almost verbatim analog for any $k$.
	
	\medskip
	
	Let $\Phi:=\{ \phi \in S(\bR^1): \supp\, \phi \subset [-1,1]\}$ and $\Psi:=\{ \psi \in S(\bR^1): \supp\, \phi \subset [1,2]\}$. Let 
	$D_\la \psi:= \la^{-1}\psi(x/\la)$. Below we list four types of bilinear operators on $\bR^k, k=1$, all of which are called paraproducts.
	$$
	\Pi_0 (f, g)(x) = \int_{\bR} \Big( (f\star D_{2^k}\psi)(g\star D_{2^k}\psi_1)\Big)\star (D_{2^k} \psi_2) \, dk,
	$$
	$$
	\Pi_1 (f, g)(x) = \int_{\bR} \Big( (f\star D_{2^k}\phi)(g\star D_{2^k}\psi_1)\Big)\star (D_{2^k} \psi_2) \, dk,
	$$
	$$
	\Pi_2 (f, g)(x) = \int_{\bR} \Big( (f\star D_{2^k}\psi_1)(g\star D_{2^k}\phi)\Big)\star (D_{2^k} \psi_2) \, dk,
	$$
	$$
	\Pi_3 (f, g)(x) = \int_{\bR} \Big( (f\star D_{2^k}\psi_1)(g\star D_{2^k}\psi_2)\Big)\star (D_{2^k} \phi) \, dk.
	$$
	Here $\phi\in \Phi, \psi, \psi_1, \psi_2\in \Psi$, and it is important that at least two functions involved in the definitions have zero integral (these are the functions from $\Psi$). These are Coifman--Meyer multipliers, for example,
	$$
	\Pi_0(f, g) = \int_{\bR^1\times \bR^1} m(\xi_1, \xi_2) \hat f (\xi_1) \hat g (\xi_2) e^{2\pi i x(\xi_1+\xi_2)} d\xi_1 d\xi_2,
	$$
	where
	$$
	m(\xi_1, \xi_2) = \int_{\bR} (2^k D_{2^k} \hat\psi)(\xi_1)(2^k D_{2^k} \hat\psi_1)(\xi_2) (2^k D_{2^k} \hat\psi_2)(-\xi_1)-\xi_2)\, d k,
	$$
	and one can check that it satisfies \eqref{CM1}.
	
	\bigskip
	
	Paraproducts is arguably one of the most important class of operators in harmonic analysis, since their boundedness properties are at the core of many problems. They appeared from PDE questions, and we discuss  now  some PDE questions which are simple to formulate, fundamental, but not easy to prove. The first one is the usual Leibniz rule for fractional derivative. Consider the {\bf fractional} derivative of order $\alpha$:
	$ f\in S,\,\, \hat \cD^{\alpha} f  (\xi) = |\xi|^{\alpha} \hat f (\xi)$. The following rule (Leibniz rule) is of paradigmatic importance in treating regularity questions for many linear and non-linear PDE (see \cite{KP}, \cite{KS}):
	\begin{equation*}
	\|\cD^\alpha(fg)\|_r \lesssim \|\cD^\alpha f\|_p \|g\|_q + \|f\|_p \|\cD^\al g\|_q,\,\, 1<p, q \le \infty, \, \frac1r =\frac1p+\frac1q, 0<r<\infty.
	\end{equation*}
	
	To prove this innocently looking inequality, one really needs paraproducts -- in fact one needs three facts, the first two being obvious: 1) $fg = \sum_{j=0}^3 \Pi_j(f, g)$, 2) $\cD^\al \Pi_j (f, g) =  \Pi_j '(f, \cD^\al g) $, where $\Pi'$ means that we replace one of $\psi$ by $\cD^{-\al} \psi$ and observe that $\cD^{-\al}\Psi\subset \Psi$ (of course this is false for the class $\Phi$), and 3) paraproducts have the boundedness property:
	
	\begin{equation}
	\label{E:pp}
	\|\Pi_j(f, g) \|_r \lesssim \| f\|_p \|g\|_q ,\,\, 1<p, q \le \infty, \, \frac1r =\frac1p+\frac1q, 0<r<\infty, j=0,1,2,3.
	\end{equation}
	The latter property is proved in \cite{MeCo}, \cite{KS}, \cite{GraTor}.
	
	It turned out that the need to study  more complicated operators, so-called multi-parameter paraproducts, appears naturally -- again primarily from PDE. They are easy to describe, but difficult to work with. Now functions $f, g$ are test functions on $\bR^{2k}$ and again for simplicity we consider only $k=1$ (having in mind that generalization to arbitrary $k$ is straightforward). Here they are:
	$$
	\Pi_j\otimes \Pi_{j'},\,\, j, j'=0, \dots, 3\,.
	$$
	Proving the boundedness of those operators of the type described in \eqref{E:pp} is much harder, but this has been done in
	\cite{MuPiTaTh1} (see also \cite{MuPiTaTh2}), and this boundedness was used  by Kenig in \cite{K} to treat well-posedness of the Kadomtsev--Petviashvili equation describing non-linear wave motion. Another type of Leibniz rule was required, and again it is reduced to the boundedness of paraproducts, this time of the mentioned above multi-parameter (tensor) type.
	
	
	Those bi-parameter paraproducts $
	\Pi_j\otimes \Pi_{j'},\,\, j, j'=0, \dots, 3
	$ (and their $n$-parameter analogs, $n\ge 3$) are the main object of our research. It has been noticed that the right model for
	studying paraproducts (usual, as in \eqref{E:pp}, or multi-parameter) are the  so-called dyadic paraproducts.
	
	In the formulas above, namely, where we defined $\Pi_j$, let us replace the continuous parameter $k$ by the discrete one: $k\in \bZ$, and let us replace  the continuous convolution by the discrete one.  We consider only $\Pi_1$ (this does not restrict the generality). We wish to keep the essential paraproduct structure, but to simplify it ``to the bare bones." So now the function $\phi$ will be the characteristic function of the unit interval, and the functions $\psi_1=\psi_2$ are functions of zero integral, supported on the unit interval, and they are $-1$ on the left half of this interval and $+1$ on its right half. So they are the so-called Haar functions, and as $D^{2^k} \psi_i$ we have $|I|^{-1/2} h_I$, where $h_I$ is $L^2$-normalized Haar functions on dyadic intervals of length $2^k$, $k\in \bZ$. So here is a usual dyadic paraproduct 
	$$
	\pi_b(f) =\pi(b, f) = \sum_{Q \in \cD} \frac{1}{|Q|^{1/2}} \left(f, \frac{\1}{|Q|^{1/2}}\right) \,(b, h_Q)\, h_Q\,,
	$$
where $(\cdot, \cdot)$ denotes the usual $L^2$ inner product.
Here $\cD$ is the lattice of dyadic intervals $Q$ in $\bR^k, k=1,$ but it can be just as well the lattice of dyadic cubes $Q$ in  $\bR^k, k>1$. The studies of those are the same, and are well studied, and of course \eqref{E:pp} is proved. 
	
	However, for our goals we need to study bi-parameter (and $n$-parameter) paraproducts. Here they are (just $\pi\otimes \pi$ as the reader will observe), and now functions $f, b$ are test functions on $\bR^2$:
	$$
	\Pi_b(f) \!\! =  \!\!  \!\! \sum_{\substack{R=I\otimes J \\ \in \cD\otimes \cD}}  \!\!  \frac1{|R|^{1/2}} \left( \!\! f, \frac{\1}{|R|^{1/2}} \!\! \right) \,(b, h_I\otimes h_J)\, h_I\otimes h_J=  \!\!  \!\! \sum_{\substack{R = \text{dyadic} \\ \text{rectangle}}} \!\!  \!\!  \frac1{|R|^{1/2}} 
	\left( \!\! f, \frac{\1}{|R|^{1/2}} \!\! \right) \,(b, h_R)\, h_R\,.
	$$
	 
	 
	 The paramount difficulty in studying such operators is the replacement of dyadic cubes by dyadic rectangles. The unweighted Lebesgue spaces boundedness properties of such bi-parameter  creatures was established in \cite{MuPiTaTh1} (and for $n$-parameter case in \cite{MuPiTaTh2}). The main unweighted inequality looks exactly as \eqref{E:pp}, it is only much harder to prove it.
	 Therefore, the unweighted Lebesgue space boundedness of $n$-parameter paraproducts (dyadic or non-dyadic) is well understood.
	 
	 \bigskip
	 
	The question of how to characterize the boundedness of weighted $n$-parameter paraproducts is the main thrust of this project.  It is well known how much weighted $L^2$-theory usually helps unweighted $L^p$-theory. This is one of our obvious motivations. But amazingly another motivation comes from problems in several complex variables. This is quite unexpected that purely dyadic problems, having nothing to do with  analyticity,   can help to solve outstanding problems in the theory of holomorphic  functions of several complex variables. 
	
	We will describe the complex analysis problem a bit later,  for now we mention that it will be a problem of characterizing measures $\mu$ on the poly-disc $\bD^n$ such that certain Hilbert spaces of analytic functions on $\bD^n$ embed boundedly into $L^2(\bD^n, \mu)$.
	
	
	It is worth mentioning that this latter problem is always present in attempts to solve the corona problem in several variables, and is also closely related to the characterization of symbols of the bounded  Hankel operator on the Hardy space on the poly-disc \cite{FL}.
	
	\medskip

	This latter problem was solved, and was a great achievement of harmonic analysis, but recently a counterexample to one of the claims in the solution has appeared, so this is still an open question, at least for now, see \cite{V}.
	
	\bigskip

	It is well known that the boundedness of $\pi_b$ is characterized by  the inclusion $b\in BMO^d(\bR^k)$, where $d$ stands for {\it dyadic}.
	In terms of the coefficients $\beta_Q:= (b, h_Q)$, this characterization is equivalent to the {\it Carleson packing  box condition}:
	\begin{equation}
	\label{C1}
	\forall P\in \cD\,\, \sum_{Q\in \cD, \, Q\subset P} \beta_Q^2 \le C\, |P|\,.
	\end{equation}
	Here the dyadic cube  $P$ stands by the name ``box,'' and it is packed by dyadic cubes that overlap because of their different size. This was invented by Carleson in the 60's and used in complex interpolation and corona results. The weighted situation was considered by Sawyer and used in the weighted theory of Calder\'on--Zygmund operators. 
	
	The weighted setting needed was the following: the input space is $f\in L^2(\bR^k, d\mu)$, the output space is unweighted $L^2$, and so the boundedness $\pi_b: L^2(\mu)\to L^2$ is equivalent to the following embedding (below $\langle f \rangle_{Q, \mu}:= \frac1{|Q|} \int_Q f \, d\mu$):
	\begin{equation}
	\label{S}
	\sum_{Q\in \cD} |\langle f\rangle_{Q, \mu}|^2\beta_Q^2 \lesssim \int_{\bR^k} |f|^2 \, d\mu\,.
	\end{equation}
	Consider cubes $Q\times [0, \ell(Q)]\subset \bR^{k=1}_+$ and $T_Q= Q\times [\ell(Q)/2, \ell(Q)]$ forming the tiling of $\bR^{k+1}_+$. We associate the sequence $\{\beta_Q\}_{Q\in \cD}$ with (any) measure $\nu$ on $\bR^{k+1}_+$ such that
	$$
	\nu(T_Q)=\beta_Q^2\,.
	$$
	Then \eqref{S} is equivalent to the boundedness of the following operator of embedding $L^2(\bR^k, \mu)$ into $L^2(\bR^{k+1}_+, \nu)$: 
	$$
	f\in L^2(\bR^k, \mu) \to \sum_{Q \in \cD} \langle f \rangle_{Q, \mu} {\1}_{T_Q}\in L^2(\bR^{k+1}_+, \nu)\,.
	$$
	It is convenient to think that we actually embed to a \textbf{tree}. In fact, it does not restrict the generality if we require $\supp f\subset Q_0$, where $Q_0$ is a fixed unit cube. Then $\{Q\}_{Q\in \cD, Q\subset Q_0}$ can be readily associated  with a tree $T$, whose root is $Q_0$. The measure $\nu$ just weights the node associated with $Q\subset Q_0$ by $\beta_Q^2$, and we are considering the embedding
	\begin{equation}
	\label{Tr1}
	f\in L^2(\bR^k, \mu) \to \{ \langle f \rangle_{Q, \mu}\}_{Q \in \cD}\in \ell^2(T, \nu),\,\, \nu\equiv \{\beta_Q\}_{Q\in T}\,.
	\end{equation}
	Sawyer found the necessary and sufficient condition for such an embedding, which is akin to the Carleson criterion \eqref{C1}:
	\begin{equation}
	\label{S1}
	\forall P\in \cD\,\, \sum_{Q\in \cD, \, Q\subset P}\Big(\frac{\mu(Q)}{|Q|}\Big)^2 \beta_Q^2 \le C\, \mu(P)\,.
	\end{equation}
	
	\bigskip
	
	However, bi- and multi-parameter harmonic analysis is notoriously more difficult than the one parameter harmonic analysis to which the previous examples belong. Its development was started in the 80's by Sun Yang Alice Chang, Jean-Lin Journ\'e and Robert Fefferman, and continued in our century by Camil Muscalu, Jill Pipher, Christoph Thile, Terry Tao. But in the weighted multi-parameter theory there are only very partial results (somewhat) analogous to  characterization \eqref{S1} for the boundedness of embedding \eqref{Tr1}  in the case of  bi- (and multi-) parameter paraproduct operator $\Pi_b$ rather than the ``usual" paraproduct  $\pi_b$.

	The difference lies in a very simple fact: the graph of dyadic rectangles lying inside the unit cube $Q_0\subset \bR^k$ is not a tree, it is a \textbf{multi-tree}. We call it $T^k$ and it is a direct product of usual dyadic trees. The root of all our future problems lies in the fact that $T^k$ has cycles: there are many ancestors on the same level for a given node $\gamma\in T^k$.
	
	Now we are interested (compare with \eqref{Tr1}) in the following embedding into multi-tree $T^k$ (we write it only for $k=2$ to keep formulas simple):
	\begin{equation}
	\label{Tr2}
	f\in L^2(\bR^2, \mu) \to \{ \langle f \rangle_{R, \mu} \}_{R=I\times J \in \cD\times \cD}\in \ell^2(T^2, \nu),\,\, \nu\equiv \{\beta_{I\times J}\}_{I\times J\in T^2}\,.
	\end{equation}
	
	The natural guess is the simple ``box" answer in the spirit of Carleson and Sawyer, that looks as follows. Again $\beta_R:= (b, h_R)^2$ -- but notice that now even in the simples case $k=2$, $R= I\times J$, $h_R= h_I\otimes h_J$, but $I, J$ are dyadic intervals of possibly {\it different sizes}. Copying \eqref{S1} we may suggest the following answer:
	\begin{equation*}
	\forall P\in \cD\times \cD\,\, \sum_{R=I\times J\in \cD\times \cD, \, R\subset P}\Big(\frac{\mu(R)}{|R|}\Big)^2 \beta_R^2 \le C\, \mu(P)\,.
	\end{equation*}
	(The analogous question for $k>2$ is now easy to emulate.) This  ``natural" answer fails miserably even if  $k=2$ and  the measure $\mu$ is just Lebesgue measure on $\bR^2$, the counterexample was constructed by Carleson, \cite{Car}, see e.g. a paper by Tao \cite{Tao}. 
	
	\medskip

	S.-Y. A. Chang found the necessary and sufficient condition for \eqref{Tr2} to be valid, she did this for any $k\ge 2$ but only for $\mu$ being Lebesgue measure on $\bR^k$. Her answer involves ``dyadic open sets".  A dyadic open set is just any finite union of dyadic rectangles. We will routinely call dyadic open sets by letter $\Om$. So here is Chang's criterion formulated for $k=2$ (of course  only for $\mu=m_2$, Lebesgue measure on $\bR^2$):
	\begin{equation}
	\label{Ch1}
	\forall \Om\,\, \text{dyadic open set},\,\, \sum_{R=I\times J\in \cD\times \cD, \, R\subset \Om} \beta_R^2 \le C\, \mu(\Om)( =C\,|\Om|)\,.
	\end{equation}
	Of course,  the Chang--Carleson condition \eqref{Ch1} is a much stronger requirement than Sawyer--Carleson box condition \eqref{S1}. And the Carleson counterexample \cite{Car}, \cite{Tao} shows that it is strictly stronger. The same criterion of Chang works for $k>2$, but again only for Lebesgue measure $\mu=m_k$. Chang's criterion led to understanding that multi-parameter $BMO$ space, studied by Chang and Robert Fefferman, is a much more subtle object than the ``usual" $BMO$. The reader may guess that product $BMO^d$ (= Chang--Fefferman $BMO^d$, \cite{ChF1})  consists of functions $b$ such that the sequence of its Haar coefficients $(b, h_R)=: \beta_R$ satisfies \eqref{Ch1}.
	
	\medskip
	
	These beautiful developments happened in the 80's, and apart from Sawyer's work \cite{Saw} there were no developments in weighted multi-parameter theory until recently. In particular, it was left totally unsettled what happens if Lebesgue measure is replaced by an arbitrary measure. The only clear thing was that -- again, using the Carleson construction from \cite{Car}, \cite{Tao} -- one can deduce that the ``natural" answer for arbitrary $\mu$ {\it cannot be in terms of Chang's dyadic open sets}:
	\begin{equation*}
	\forall \Om\,\, \text{ dyadic open set},\,\, \sum_{R=I\times J\in \cD\times \cD, \, R\subset \Om} \Big(\frac{\mu(R)}{|R|}\Big)^2\beta_R^2 \le C\, \mu(\Om)\,.
	\end{equation*}
	Counterexamples to this ``natural criterion of embedding on the bi-tree" were constructed in \cite{MPV}.

	\bigskip
	
	Meanwhile, the need for a criterion for an embedding of type \eqref{Tr2} (not only for $k=2$ but also for all $k>2$) was initiated to large extent by natural questions from several complex variables theory, see Section \ref{planar}.


\section{Weighted multi-parameter embeddings}
\label{para}
\subsection{Background}
\label{sec:background}
Lennart Carleson showed in \cite{Car} that the natural generalization, using a ``box'' condition, from the one parameter case (disc) to the bi-parameter case (bi-disc) of his embedding theorem does not work.
Sun-Yang A.~Chang in \cite{Ch79} found the necessary and sufficient condition for the validity of the Carleson embedding for bi-harmonic extensions into the bi-disc.

The discrete versions of these results can be motivated by considering a bi-parameter dyadic paraproduct.
For a dyadic rectangle $R = I \times J \subseteq [0,1]^{2}$ denote by $h_{R}(x,y) = h_{I}(x) h_{J}(y)$ an associated $L^{2}$ normalized Haar function.
The simplest example of a \emph{bi-parameter dyadic paraproduct} is the operator
\[
\Pi_b \vf := \sum_{R}\La \vf\Ra_R (b, h_R) h_R\,.
\]
The paraproduct $\Pi_{b}$ is a bounded operator on $L^{2}$ with respect to the Lebesgue measure $m$ on $[0,1]^{2}$ if and only if we have
\begin{equation}
\label{para2}
\sum_{R}\La \vf\Ra_R^2\, \beta_R^2 \le C \int\vf^2 dm,
\end{equation}
where $\beta_R := (b, h_R)$ are Haar coefficients of the function $b$.
In analogy to the one-parameter Carleson embedding one could ask whether \eqref{para2} is equivalent to the ``box'' condition
\begin{equation}
\label{box1}
\sum_{R\subseteq R_0} \beta_R^2\le C' m_2(R_0)= C' |R_0|
\end{equation}
for every dyadic rectangle $R_{0} \subseteq [0,1]^{2}$.
A counterexample showing that \eqref{box1} does not imply \eqref{para2} was constructed by Carleson \cite{Car,Tao}.

It was observed by Chang \cite{Ch79} (in a continuous setting) that \eqref{para2} is equivalent to the \emph{bi-parameter Carleson (or Carleson--Chang) condition}
\begin{equation}
\label{Carpara1}
\sum_{R\subset \Om} \beta_R^2\le C' m_2(\Om) =C' |\Om|\,,
\end{equation}
where the constant $C'$ is uniform for all subsets $\Omega \subseteq [0,1]^{2}$ that are finite unions of dyadic rectangles.
This necessary and sufficient condition was later used by Chang and Fefferman \cite{ChF1} to characterize the dual of the Hardy space on the bi-disc $H^1(\bD^2)$.
In Appendix~\ref{sec:Chang} we recall the proof of the equivalence between \eqref{para2} and \eqref{Carpara1} in the more general setting of bi-trees.

The present article treats a certain two weight problem about bi-parameter paraproduct operators.
Singular bi-parameter operators enjoyed and continue to enjoy much attention, see \cite{RF,RF1,RF2}, \cite{P}, \cite{JLJ}, \cite{BP}.
They are notoriously difficult.
Two weight problems for singular integrals were studied in a series of papers by Nazarov, Treil, and Volberg on dyadic singular operators and in a series of papers by Lacey, Shen, Sawyer, and Uriarte-Tuero on the Hilbert transform, see  \cite{NTV99}, \cite{NTV08}, \cite{LSSUT}, \cite{L}, and the references therein.
Another example is a very recent paper by Iosevich, Krause, Sawyer, Taylor, and Uriarte-Tuero \cite{IKSTUT} on the two weight problem for the spherical maximal operator motivated by Falconer's distance set problem.

Classically, in one parameter case,  an estimate of paraproduct tri-linear forms \cite{GraTor} is based on $T1$ theorem of David and Journ\'e. The theory of Carleson measures (or classical $BMO$ theory) is involved.

\bigskip

It is well known \cite{ChF1, ChF2, JLJ, JLJ2} that in the multi-parameter setting all these results and concepts of Carleson measure, $BMO$, John--Nirenberg inequality, Calder\'on--Zygmund decomposition are much more delicate. Paper \cite{MPTT1}  develops a completely new approach to prove natural tri-linear bi-parameter estimates on bi-parameter paraproducts, especially outside of Banach range. In \cite{MPTT1} Journ\'e's lemma \cite{JLJ2} was used, but the approach did not generalize to multi-parameter paraproduct forms. This issue was resolved in \cite{MPTT2}, where a simplified method was used  to address the multi-parameter paraproducts.

In the present paper we deal with the boundedness of bi-parameter paraproducts from one weighted $L^2$ space to another. One of the weight is completely arbitrary, which is dictated by the pedigree of our problem as the embedding problem for a class of holomorphic functions in bi-disc. 
We hope that our approach would work for higher dimensional poly-discs. There is a partial confirmation of that hope in \cite{MPVZ}, where we treated tri-disc case.

\subsection{Summary}
In the works mentioned in Section~\ref{sec:background} the ``underlying'' measure was the Lebesgue measure on bi-torus, and the ``embedding'' measure on the bi-disc was a priori arbitrary.
In this article we switch the constraints on two measures involved:  the ``underlying'' measure on the bi-torus is arbitrary, while the ``embedding'' measure on the bi-disc has a special structure.
This problem appeared in trying to understand the embedding of Dirichlet space of holomorphic functions on bi-disc \cite{AMPS18}, this connection will be discussed in more detail in Section~\ref{planar}.
For uniform ``embedding'' measures several necessary and sufficient conditions of Carleson--Chang type for a Carleson embedding of Dirichlet space on the bi-disc were found in \cite{AMPS18, AHMV18b}.

In this article we show the unexpected fact that for ``embedding'' measures with product structure a condition of the ``box'' type \eqref{box1} turns out to characterize the Carleson embedding~\eqref{para2}.
This is a striking difference to the case considered by Chang and Fefferman, where Carleson's counterexample explicitly forbids for such an effect to happen.

We have arbitrary measure on bi-torus and special measures on bi-disc (dictated by the norm structure of Dirichlet spaces), Chang--Fefferman reversed the roles of measures, the one on bi-torus was only Lebesgue measure, the one on bi-disc was arbitrary.  The reader may ask: what is both measures are arbitrary? The answer is that 
we constructed several counterexamples that break down all natural conjectures for arbitrary case, see  \cite{MPV}.

\subsection{Terminology and notation}
\label{sec:notation}
We begin with order-theoretic conventions.
\begin{definition}
A \emph{finite tree} $T$ is a finite partially ordered set such that for every $\omega\in T$ the set $\{\alpha\in T \colon \alpha\geq\omega\}$ is totally ordered (we allow trees to have several maximal elements).

A \emph{bitree} $T^{2}$ is a cartesian product of $2$ finite trees with the product order.

A subset $\mathcal{U}$ (resp.\ $\mathcal{D}$) of a partially ordered set $T$ is called an \emph{up-set} (resp.\ \emph{down-set}) if for every $\alpha\in\mathcal{U}$ and $\beta\in T$ with $\alpha \leq \beta$ (resp.\ $\beta \leq \alpha$) we also have $\beta\in\mathcal{U}$  (resp.\ $\beta\in\mathcal{D}$).
\end{definition}

The \emph{Hardy operator} on the bitree $T^{2}$ is defined by
\begin{equation}
\label{eq:Hardy}
\bI \phi (\gamma) := \sum_{\gamma'\ge \gamma} \phi(\gamma')
\quad\text{for any } \phi: T^{2}\to \bR.
\end{equation}
In the one-parameter case $T^{1}=T$ we denote it by $I$.
The adjoint $\bI^*$ of the Hardy operator $\bI$ is given by the formula
\begin{equation}
\label{eq:adj-Hardy}
\bI^* \psi (\gamma)= \sum_{\gamma'\le \gamma} \psi(\gamma').
\end{equation}
\begin{definition}
Let $\mu,w$ be positive functions on $T^{2}$.
The \emph{box constant} is the smallest number $[w,\mu]_{Box}$ such that
\begin{equation}
\label{eq:box}
\mathcal{E}_{\beta}[\mu] :=
\sum_{\alpha \leq \beta} w(\alpha) (\bI^{*}\mu(\alpha))^{2}
\leq [w,\mu]_{Box}
\sum_{\alpha \leq \beta} \mu(\alpha),
\quad \forall \beta\in T^{2}.
\end{equation}
The \emph{Carleson constant} is the smallest number $[w,\mu]_{C}$ such that
\begin{equation}
\label{eq:Carleson}
\sum_{\alpha\in\mathcal{D}} w(\alpha) (\bI^{*}\mu(\alpha))^2 \leq [w,\mu]_{C} \mu(\mathcal{D}),
\quad \forall \mathcal{D}\subset T^{2} \text{ down-set.}
\end{equation}
The \emph{hereditary Carleson constant (or restricted energy condition constant or REC constant)} is the smallest constant $[w,\mu]_{HC}$ such that
\begin{equation}
\label{eq:hereditary-Carleson}
\mathcal{E}[\mu \one_{E}]
=
\sum_{\alpha\in T^{2}} w(\alpha) (\bI^{*} (\mu\one_{E})(\alpha))^2
\le [w,\mu]_{HC} \mu(E),
\quad \forall E \subset T^{2}.
\end{equation}
The \emph{Carleson embedding constant} is the smallest constant $[w,\mu]_{CE}$ such that the adjoint embedding
\begin{equation}
\label{bIstar}
\sum_{\alpha\in T^{2}} w(\alpha) \abs{\bI^*( \psi\mu)(\alpha) }^2
\leq [w,\mu]_{CE} \sum_{\omega \in T^{2}} \abs{\psi(\omega)}^2 \mu (\omega)
\end{equation}
holds for all functions $\psi$ on $T^{2}$.
\end{definition}
In order to make a connection to the objects in Section~\ref{sec:background} consider the bi-tree $T^{2}$ that consists of the dyadic rectangles in $[0,1]^{2}$ with side lengths between $1$ and $2^{-N}$, where $N$ is a large finite integer.
Let
\begin{equation}
\label{eq:mu-on-dyadic-tree}
\mu(\omega) =
\begin{cases}
m(\omega) & \text{if } \omega \text{ is a } (2^{-N}\times 2^{-N}) \text{-square},\\
0 & \text{otherwise}.
\end{cases}
\end{equation}
Suppose also $\beta_R^2 =  m(R)^2 w_R$.
Then \eqref{Carpara1} is equivalent to \eqref{eq:Carleson} with constant independent of $N$, while \eqref{para2} is equivalent to \eqref{bIstar} with constant independent of $N$.
The box condition~\eqref{box1} is equivalent to \eqref{eq:box}, again with constant independent of $N$. All these just mentioned  equivalences hold for Lebesgue measure $\mu$ as in \eqref{eq:mu-on-dyadic-tree}. In the present article we will describe all embedding measures $\mu$ such that \eqref{bIstar} is valid if $w$ has product structure.

For positive numbers $A,B$ we write $A \lesssim B$ if $A\leq CB$ with an absolute constant $C$, that in particular does not depend on the tree or bi-tree or the weights $w,\mu$.

\subsection{Main result}
The inequalities
\begin{equation}
\label{eq:obvious}
[w,\mu]_{Box} \le [w,\mu]_{C} \leq [w,\mu]_{HC} \leq [w,\mu]_{CE}
\end{equation}
are obvious.
It turns out that for product measures $w$ there are also converse inequalities.
\begin{theorem}
\label{thm:main}
Let $\mu,w$ be positive measures on $T^{2}$.
Assume that $w$ is of the product form
\begin{equation}
\label{eq:w-product}
w(\alpha)
= w(\alpha_{x},\alpha_{y})
= w_{x}(\alpha_{x})w_{y}(\alpha_{y}).
\end{equation}
Then the reverses of the inequalities in \eqref{eq:obvious} also hold:
\[
[w,\mu]_{CE}
\lesssim [w,\mu]_{HC}\lesssim [w,\mu]_{C} \lesssim [w,\mu]_{Box}.
\]
\end{theorem}
\begin{remark}
In Carleson's counterexample $\mu=m$, but the coefficients $\beta_R$ are not of the form $m(R)^2 w_R$ with $w_R$ having product structure.
The leading particular case of Theorem~\ref{thm:main} is $w\equiv 1$, which corresponds to the description of all embedding measures for the unweighted Dirichlet space on the bi-disc.
Already in this case most combinatorial features of the proof for general weights appear.
\end{remark}
The index $2$ in the notation $T^{2}$ stands for the fact that the tree has two parameters.
In particular we allow $T^{2} = T\times T'$, where $T$ and $T'$ are different simple trees.
When the second tree $T'$ consists of one point, Theorem~\ref{thm:main} recovers the two weight Carleson embedding in \cite{NTV99}.
Alternative proofs of that result can be found in \cite{AHMV18a} and in Appendix~\ref{sec:Chang}.
In the case $w\equiv 1$ equivalence of the Carleson~\eqref{eq:Carleson} and the Carleson embedding~\eqref{bIstar} conditions on (dyadic) bi-trees was proved in \cite{AMPS18}, and in addition an equivalent capacitary condition was found in that article.
An alternative proof of this equivalence in the case $w\equiv 1$ was given in \cite{AHMV18b}. 

But neither \cite{AMPS18} nor \cite{AHMV18b} proved the equivalence of embedding and box condition \eqref{eq:box} even for $w\equiv 1$. This is partially because in view of Carleson's counterexample it was hard to imagine the equivalence of box condition and embedding.

We would also like to mention a result of Sawyer \cite{Saw}, who obtained a complete description of the two-weight embedding for a special (hooked) weight $w$.
We will discuss Sawyer's result in more detail in Section~\ref{sec:Sawyer}.

The main contribution of the present article is the unexpected equivalence of the box condition~\eqref{eq:box} with the other conditions listed above, see Theorem~\ref{thm:box=>HC}.
The extension of the previous results to product weights $w$ is also new.
This extension follows the argument in \cite{AHMV18b}, which we streamline considerably in Theorems~\ref{thm:C-to-HC} and \ref{thm:HC-to-CE}.

The second part of this article contains counterexamples with general $w$ to all equivalences proved in the case of product $w$.
One such example is of course due to Carleson \cite{Car}, but the weight $w=\{w_R\}$ is not very simple in that example.
We will give examples in which for each dyadic rectangle $R$ the weight $w_R$ is either $0$ or $1$.
This corresponds to embeddings on \emph{subgraphs} of bi-trees.

The question of embedding to weighted $k$-trees, $k\ge 3$, their subgraphs, and more general graphs with cycles remains wide open.
The difference between bi-parameter singular integrals  and $k$-parameter, $k\ge 3$, is not new for multi-parameter theory.
See, for example, \cite{JLJ}, where J.-L.~Journ\'e constructed a $3$-parameter singular Calder\'on--Zygmund operator $S$ mapping (as it should) $L^2(\bR^3)$ to itself, but not mapping $L^\infty$ to product (Chang--Fefferman) $BMO$ even though $S$ acts in the expected way on the rectangular atoms introduced in \cite{RF1}.
We are grateful to Jill Pipher who indicated to us the paper \cite{JLJ}.
The counterexample in \cite{JLJ} is related to the counterexample of Carleson \cite{Car} showing that rectangular Carleson measures are not, on the bi-disc, a good substitute for classical Carleson measures.

\subsection{Applications to complex analysis on the bi-disc}
\label{planar}
As already mentioned, the results of Carleson, Chang, and Fefferman correspond to the situation where the measure $\mu$ is fixed to be the Lebesgue measure, and the weight $w$ is arbitrary.
Here we impose a condition (product structure) on the weight (for example we can consider $w\equiv 1$), and vary the measure.
One reason to consider such a problem stems from the problems about Carleson measures for a certain scale of Hilbert spaces of analytic functions on the bi-disc.
Namely, given $s = (s_1,s_2)\in\mathbb{R}^2$ consider the space $\mathcal{H}_s$ of analytic functions $f$ on the unit bi-disc for which the norm
\[
\sum_{n_1,n_2\geq0}(n_1+1)^{s_1}(n_2+1)^{s_2} \abs{\hat{f}(n_1,n_2)}^{2} =: \norm{f}^2_{s}
\]
is finite.
One then can ask for which measures $\mu$ the embedding $\mathcal{H}_s \rightarrow L^2(\mathbb{D}^2,d\mu)$ is bounded. 
The Hardy space on the bi-disc (i.e., $\mathcal{H}_{(0,0)}$) corresponds to the Carleson--Chang--Fefferman case.
It turns out that the special form of the second case, arbitrary $\mu$ and $w \equiv 1$, describes the embedding of the Dirichlet space $\mathcal{H}_{(1,1)}$ on the bi-disc (see \cite{AMPS18}).
Moreover, adjusting the pair $(\mu,w)$, one can get a boundedness criterion for the embedding of any space $\mathcal{H}_s$ on the Hardy--Dirichlet scale on the bi-disc.

\subsection{Combinatorial property of all planar positive measures}
\label{allmeasures}
The equivalence 
\[
[\one ,\mu]_{Box} \asymp [\one,\mu]_{CE}
\]
still amazes us.
More explicitly, the essence of this equivalence is that for any positive measure $\mu$ on a square $Q_0$ the property that
\[
\sum_{R\subset R_0} \mu(R)^2 \le \mu(R_0)\quad \text{for all dyadic rectangles} \,\, R_0
\]
implies
\[
\sum_{R\subset \cup_i R_i} \mu(R)^2 \le \mu(\cup_i R_i)\quad \text{for all  unions dyadic rectangles} \,\, R_i\,.
\]
This contradicts---at first glance---the examples of Carleson type \cite{Car,Tao}, and goes against the general feeling that huge overlap among the dyadic rectangle must prevent this equivalence from happening. But it happens, see below.

\section{Small energy majorization on an ordinary tree}
\label{T1}

\begin{definition}
\label{def:superadditive}
Given a finite tree $T$, the set of \emph{children} of a vertex $\beta\in T$ consists of the maximal elements of $T$ that are strictly smaller than $\beta$:
\[
\ch \beta := \max \Set{ \beta' \in T \colon \beta' < \beta}
\]
A function $g : T \to [0,\infty)$ is called \emph{superadditive} if for every $\beta \in T$ we have
\[
g(\beta) \geq \sum_{\beta' \in \ch(\beta)} g(\beta').
\]
\end{definition}

\begin{lemma}
\label{lem:supadditive-l1linf}
Let $T'$ be a finite tree and $g,h : T' \to [0,\infty)$.
Assume that $g$ is superadditive and $Ih \leq \lambda$ on $\supp g$.
Then for every $\beta \in T'$ we have
\[
\sum_{\alpha \leq \beta} g(\alpha) h(\alpha)
\leq
\lambda g(\beta).
\]
\end{lemma}
\begin{proof}
Without loss of generality we may consider the case when $\beta$ is the unique maximal element of $T'$ and $T' = \supp g$.
We induct on the depth of the tree.
Let $T'$ be given and suppose that the claim is known for all its branches.
Then by the inductive hypothesis and superadditivity we have
\begin{align*}
\sum_{\alpha \leq \beta} g(\alpha) h(\alpha)
&=
g(\beta) h(\beta) + \sum_{\beta' \in \ch(\beta)} \sum_{\alpha \leq \beta'} g(\alpha) h(\alpha)
\\ &\leq
g(\beta) h(\beta) + \sum_{\beta' \in \ch(\beta)} g(\beta') \sup_{\alpha \leq \beta'} \sum_{\alpha \leq \alpha' \leq \beta'} h(\alpha')
\\ &\leq
g(\beta) h(\beta) + \sum_{\beta' \in \ch(\beta)} g(\beta') \sup_{\alpha < \beta} \sum_{\alpha \leq \alpha' < \beta} h(\alpha')
\\ &\leq
g(\beta) h(\beta) + g(\beta) \sup_{\alpha < \beta} \sum_{\alpha \leq \alpha' < \beta} h(\alpha')
\\ &=
g(\beta) \sup_{\alpha \leq \beta} \sum_{\alpha \leq \alpha' \leq \beta} h(\alpha').
\qedhere
\end{align*}
\end{proof}

\begin{lemma}
\label{lem:I2-positive}
Let $I$ be an integral operator with a positive kernel and $f,g$ positive functions.
Then
\[
\int (If)^{2} g \leq \Bigl( \sup_{\supp g} II^{*}g \Bigr) \int f^{2}.
\]
\end{lemma}
\begin{proof}
Without loss of generality $f$ is positive.
By duality we have
\[
\int (If)^{2} g
=
\int f I^{*}(If \cdot g)
\leq
\norm{f}_{2} \norm{I^{*}(If \cdot g)}_{2}.
\]
By the hypothesis $Ih(x) = \int K(x,y) h(y)$ with a positive kernel $K$.
Hence
\begin{align*}
\norm{I^{*}(If \cdot g)}_{2}^{2}
&=
\int I^{*}(If \cdot g) I^{*}(If \cdot g)
\\ &=
\int K(x,y) ((If)(x) g(x)) K(x',y) ((If)(x') g(x')) \dif(x,x',y)
\\ &\leq
\int \frac12 (If(x)^{2}+If(x')^{2}) K(x,y) (g(x)) K(x',y) (g(x')) \dif(x,x',y)
\\ &=
\frac12 \int I^{*}((If)^{2} \cdot g) I^{*}(g) + \int I^{*}(g) I^{*}((If)^{2} \cdot g)
\\ &=
\int (II^{*}g) \cdot (If)^{2} \cdot g
\\ &\leq
\Bigl( \sup_{\supp g} II^{*}g \Bigr) \int (If)^{2} \cdot g.
\end{align*}
Substituting the second displayed estimate into the first we obtain
\[
\int (If)^{2} g
\leq
\norm{f}_{2} \Bigl( \sup_{\supp g} II^{*}g \Bigr) \Bigl( \int (If)^{2} \cdot g \Bigr)^{1/2}.
\]
The conclusion follows by rearranging the terms.
\end{proof}

\begin{lemma} 
\label{Phi}
Let $g,f,w : T \to [0,\infty)$ be positive functions and $\lambda,\delta>0$.
Assume that $g$ is superadditive and $I (wg) \leq \delta$ on $\supp f$.
Then there exists a positive function $\phi : T \to [0,\infty)$ such that
\begin{equation} 
\label{geIf}
I (w\phi) \gtrsim I(wf) \text{ on } \{ \lambda/2 < I(wg) \leq 2\lambda\},
\end{equation}
and
\begin{equation} \label{enest1}
\int w \phi^{2} \lesssim \frac{\delta}{\lambda} \int w f^{2}.
\end{equation}
\end{lemma}

\begin{proof}
Without loss of generality we may assume $\lambda \geq 4\delta$.
Define
\begin{equation}
\label{PhiDef}
\phi(\alpha) :=
\frac{1}{\lambda} \one_{\delta < I(wg)(\alpha) \leq 2\lambda} I(wf)(\alpha) g(\alpha)
\end{equation}

We prove first \eqref{geIf}.
Let $\omega\in T$ be such that $\lambda/2 < I(wg)(\omega) \leq 2\lambda$.
Then for every $\alpha\in T$ with $\alpha\geq\omega$ and $I(wg)(\alpha)>\delta$ we have
\[
I(wf)(\alpha) = I(wf)(\omega).
\]
It follows that
\begin{align*}
I(w\phi)(\omega)
&=
\frac{1}{\lambda} \sum_{\substack{\alpha\geq\omega :\\ \delta < I(wg)(\alpha) \leq 2\lambda}} I(wf)(\alpha) (wg)(\alpha)
\\ &=
I(wf)(\omega) \frac{1}{\lambda} \sum_{\substack{\alpha\geq\omega :\\ \delta < I(wg)(\alpha)}} (wg)(\alpha)
\\ &=
I(wf)(\omega) \frac{1}{\lambda} ( I(wg)(\omega) - I(wg)(\alpha_{min}) ),
\end{align*}
where $\alpha_{min}$ is the smallest $\alpha$ outside of the summation range if it exists (otherwise that term is omitted).
But then $I(wg)(\omega) \geq \lambda/2$ and $I(wg)(\alpha_{min}) \leq \delta$, and \eqref{geIf} follows.

Next we will prove the energy estimate \eqref{enest1}.
Let $\mathcal{U} := \{I(wg) \leq \delta\}$, so that $\mathcal{U}$ is an up-set and $f$ is supported on $\mathcal{U}$.
By Lemma~\ref{lem:I2-positive} with the operator $I \sqrt{w} \one_{\mathcal{U}}$ and functions $f \sqrt{w}$ and $g^{2} w \one_{I(wg)\leq 2\lambda}$ we can estimate
\begin{align*}
\int w \phi^{2}
&\leq
\frac{1}{\lambda^{2}} \sum_{\substack{\alpha :\\ I(wg)(\alpha) \leq 2\lambda}} I(wf)(\alpha)^{2} g(\alpha)^{2} w(\alpha)
\\ &\leq
\frac{1}{\lambda^{2}} \Bigl( \int w f^{2} \Bigr) \sup I(w \one_{\mathcal{U}} I^{*}(g^{2} w \one_{I(wg)\leq 2\lambda})).
\end{align*}
By Lemma~\ref{lem:supadditive-l1linf} with the superadditive function $g \one_{I(wg) \leq 2\lambda}$ and the function $h=wg$ we can estimate
\[
I^{*}(g^{2} w \one_{I(wg)\leq 2\lambda}) \leq 2\lambda g.
\]
Moreover, since $\mathcal{U}$ is an up-set on a simple tree we have
\[
I (w \one_{\mathcal{U}} g) \leq \sup_{\mathcal{U}} I(wg) \leq \delta.
\]
Combining the last three displays we obtain the energy estimate \eqref{enest1}.
\end{proof}

\section{Small energy majorization on a bitree}
\label{T2lemma}

\begin{theorem}
\label{2D1}
Let $m : T^{2} \to [0,\infty)$ be a positive function and $\lambda,\delta>0$.
Let $w$ be a positive measure on $T^{2}$ that is of the product form \eqref{eq:w-product}.
Assume that $m$ is superadditive and that $\bI (w m) \leq \delta$ on $\supp m$.
Then there exists a positive function $\phi$ on $T^{2}$ such that
\begin{align}
\label{eq:lower-bound}
\bI(w \phi) &\gtrsim \bI (w m) \text{ on the set } \{\lambda < \bI(w m) \leq 2\lambda\}, \text{ and}
\\
\label{eq:small-energy}
\int w \phi^2 &\lesssim \frac{\delta}{\lambda} \int w m^{2}.
\end{align}
\end{theorem}

\begin{proof}[Proof of Theorem~\ref{2D1}]
We may assume $\lambda \gg \delta$.
For a fixed $\alpha_y \in T_y$ we will apply Lemma~\ref{Phi} data $(g^{\alpha_{y}},f^{\alpha_{y}},w_{x},\lambda)$ that we will now construct.
Let
\[
g^{\alpha_{y}}(\beta_x)
:= \sum_{\alpha_y' \ge \alpha_y} m(\beta_x\times \alpha_y') w_{y}(\alpha_{y}'),
\quad \beta_{x} \in T_{x}.
\]
Superadditivity of $m$ implies that $g^{\alpha_{y}}$ is superadditive on $T_{x}$.
Also
\begin{equation}
\label{B}
I_{x} (w_{x} g^{\alpha_{y}})(\beta_x)
= \sum_{\beta_x'\ge \beta_x, \alpha_y'\ge \alpha_y} w_{x} m(\beta_x'\times \alpha_y') w_{y}(\alpha_{y}')
= \bI (w_{x}w_{y}m) (\beta_x\times \alpha_y),
\end{equation}
where $I_{x}$ denotes the Hardy operator on the tree $T_{x}$.
Let
\[
f^{\alpha_{y}}(\beta_x) :=
m (\beta_x\times \alpha_y).
\]
From \eqref{B} we see that $I_{x} (w_{x} g^{\alpha_{y}}) \leq \delta$ on $\supp f^{\alpha_{y}}$.
By Lemma~\ref{Phi} with data $(g^{\alpha_{y}},f^{\alpha_{y}},w_{x},\lambda)$ we obtain a function $\phi^{\alpha_y} : T_x \to [0,\infty)$ such that
\begin{equation}\label{vf1}
I_{x} (w_{x} \phi^{\alpha_y})(\omega_x) \gtrsim
I_{x} (w_{x}f^{\alpha_y})(\omega_x) \text{ if } \lambda/2 < I_{x} (w_{x}g^{\alpha_{y}})(\omega_{x}) \leq 2\lambda,
\end{equation}
\begin{equation}\label{vf2}
\int_{T_{x}} w_{x} (\phi^{\alpha_y})^2
\lesssim \frac{\delta}{\lambda} \int_{T_{x}} w_{x} (f^{\alpha_y})^2.
\end{equation}
It is nice to have the following formulas in mind (the second one follows from \eqref{PhiDef})
\[
g^{\alpha_{y}}(\beta_x) = I_y [w_y f^{\alpha_{y}}(\beta_x)]
\]
\begin{align*}
&\phi^{\alpha_y}(\beta_x)=\frac{1}{\lambda} \one_{\delta < I(wg)(\alpha) \leq 2\lambda} I_x [w_x f^{\alpha_y}](\beta_x)g^{\alpha_y}(\beta_x)
\\
& =\frac{1}{\lambda} \one_{\delta < I(wg)(\alpha) \leq 2\lambda} I_x [w_x f^{\alpha_y}](\beta_x)  I_y [w_y f^{\alpha_{y}}(\beta_x)]
\end{align*}

Now put
\[
\phi(\beta_x \times \alpha_y) := \phi^{\alpha_y}(\beta_x).
\]
Summing \eqref{vf2} over all $\alpha_y \in T_y$ we obtain
\begin{align*}
\int_{T} w_{x} w_{y} \phi^2
&=
\sum_{\alpha_y} w_{y}(\alpha_{y}) \int_{T_{x}} w_{x} (\phi^{\alpha_{y}})^{2}
\\ &\lesssim
\frac{\delta}{\lambda} \sum_{\alpha_y} w_{y}(\alpha_{y}) \int_{T_{x}} w_{x} (f^{\alpha_{y}})^{2}
\\ &=
\frac{\delta}{\lambda} \int_{T} w_{x}w_{y} m^{2}.
\end{align*}
This shows~\eqref{eq:small-energy}.

Let now $\omega=\omega_x\times\omega_y \in T^{2}$ with $\lambda < \bI (w_{x}w_{y}m)(\omega) \leq 2 \lambda$.
Summing \eqref{vf1} in $\alpha_y$ we obtain
\begin{align*}
\bI (w_{x}w_{y}\phi) (\omega)
&\geq
\sum_{\substack{\alpha_y\geq\omega_y :\\ \lambda/2 < \bI(w_{x}w_{y}m)(\omega_{x} \times \alpha_{y}) \leq 2\lambda}} I_{x} (w_{x} \phi^{\alpha_y})(\omega_x) w_{y}(\alpha_{y})
\\ &\gtrsim
\sum_{\substack{\alpha_y\geq\omega_y :\\ \lambda/2 < \bI(w_{x}w_{y}m)(\omega_{x} \times \alpha_{y})}} I_{x} (w_{x} f^{\alpha_y})(\omega_x) w_{y}(\alpha_{y})
\\ &=
\bI (w_{x} w_{y} m)(\omega_x \times \omega_{y}) - \bI (w_{x} w_{y} m)(\omega_x \times \alpha_{y,min}),
\end{align*}
where $\alpha_{y,min}$ is the smallest index outside the summation range if it exists (otherwise this term is omitted).
But then $\bI (w_{x} w_{y} m)(\omega_x \times \omega_{y}) > \lambda$ and $\bI (w_{x} w_{y} m)(\omega_x \times \alpha_{y,min}) \leq \lambda/2$.
This shows~\eqref{eq:lower-bound}.
\end{proof}

\section{Main lemma}
Define
\[
\bV^{\mu} := \bI( w \bI^{*} \mu),
\]
\[
E_\delta := \{ \bV^{\mu} \leq \delta \} \subseteq T^{2},
\]
\[
\bV_\delta^{\mu} := \bI ( w \one_{E_{\delta}} \bI^{*} \mu ),
\]
\[
\mathcal{E}_\delta[\mu]
:=
\int_{T^{2}} \bV_\delta^{\mu} \dif\mu
=
\int_{E_\delta} w (\bI^{*} \mu )^2,
\]
\[
\mathcal{E}[\mu]
:=
\int_{T^{2}} \bV^\mu \dif\mu
=
\sum_{\alpha\in T^{2}} w(\alpha) ((\bI^*\mu)(\alpha))^2.
\]

\begin{lemma}
\label{lem:cEcE}
Let $\mu,\rho,w$ be positive measures on $T^{2}$ and $\delta>0$.
Assume that $w$ is of the product form~\eqref{eq:w-product}.
Then
\begin{equation}
\label{eq:cEcE}
\Bigl( \int \bV_{\delta}^{\mu} \dif\rho \Bigr)^{3}
\lesssim
\delta \mathcal{E}_\delta[\mu] \mathcal{E}[\rho] \abs{\rho}.
\end{equation}
\end{lemma}

\begin{remark}
On an ordinary tree Lemma~\ref{lem:cEcE} is trivial because in this case $\bV^{\mu}_{\delta} \leq \delta$, so \eqref{eq:cEcE} can be recovered from the pair of estimates
\[
\int \bV_{\delta}^{\mu} \dif\rho
\leq
\delta \abs{\rho},
\quad
\int \bV_{\delta}^{\mu} \dif\rho
\leq
\mathcal{E}_\delta[\mu]^{1/2} \mathcal{E}[\rho]^{1/2}.
\]
On a bitree the potential $\bV_{\delta}^\mu$ can be considerably bigger than $\delta$.
\end{remark}

\begin{proof}
Without loss of generality $\mathcal{E}_{\delta}[\mu]\neq 0$ and $\rho\not\equiv 0$.
Let $\lambda>0$ be chosen later.

Let $m(\alpha) := \bI^{*}\mu(\alpha) \one_{E_{\delta}}(\alpha)$.
Since $E_{\delta}$ is an up-set, this function is superadditive.
Also, $\bI (wm) = \bV_{\delta}^{\mu} \leq \bV^{\mu} \leq \delta$ on $\supp m$, and $\mathcal{E}_{\delta}[\mu] = \int w m^{2}$.
By Theorem~\ref{2D1} with data $(g,f,2^{n}\lambda,\delta)$ we obtain functions $\phi_{n}$ such that $\bI (w\phi_{n}) \geq \bI (wm)$ on the set $\{ 2^{n}\lambda < \bI (wm) \leq 2^{n+1}\lambda \}$ and $\int w \phi_{n}^{2} \lesssim \delta/(2^{n} \lambda) \cdot \int w m^{2}$.
Then
\begin{align*}
\int \bV^\mu_\delta \dif\rho
&=
\int_{\{\bV^\mu_{\delta} \leq \lambda\}} \bV^\mu_{\delta} \dif\rho
+ \sum_{n=0}^{\infty} \int_{\{2^{n}\lambda < \bV_{\delta}^{\mu} \leq 2^{n+1}\lambda\}} \bV^\mu_\delta \dif\rho
\\ &\le
\lambda \abs{\rho}
+ \sum_{n=0}^{\infty} \int_{\{ 2^{n}\lambda < \bV_{\delta}^{\mu} \leq 2^{n+1}\lambda\}} \bI (w\phi_{n}) \dif\rho
\\ &\leq
\lambda \abs{\rho}
+ \sum_{n=0}^{\infty} \bigl( \int w \phi_{n}^{2} \bigr)^{1/2} \mathcal{E}[\rho]^{1/2}
\\ &\leq
\lambda \abs{\rho}
+ \sum_{n=0}^{\infty} C (\delta/(2^{n}\lambda))^{1/2} \mathcal{E}_{\delta}[\mu]^{1/2} \mathcal{E}[\rho]^{1/2}.
\\ &\leq
\lambda \abs{\rho}
+ C (\delta/\lambda)^{1/2} \mathcal{E}_{\delta}[\mu]^{1/2} \mathcal{E}[\rho]^{1/2}.
\end{align*}
Substituting $\lambda = (\delta \mathcal{E}_{\delta}[\mu] \mathcal{E}[\rho])^{1/3} \abs{\rho}^{-2/3}$ we obtain \eqref{eq:cEcE}.
\end{proof}
\begin{remark}
By the Cauchy--Schwarz inequality we have the trivial estimate
\[
\int \bV_{\delta}^{\mu} \dif\rho
\leq
\mathcal{E}_\delta[\mu]^{1/2} \mathcal{E}[\rho]^{1/2}.
\]
Lemma~\ref{lem:cEcE} improves upon this estimate if $\delta \abs{\rho} \lesssim \mathcal{E}_\delta[\mu]^{1/2} \mathcal{E}[\rho]^{1/2}$, which is equivalent to $\lambda/\delta \gtrsim 1$ in the proof.
\end{remark}

\section{Carleson implies hereditary Carleson}
\label{CtoHC}
Taking $\rho=\mu$ in Lemma~\ref{lem:cEcE} we obtain the following result.
\begin{corollary}
\label{cor:cEcE}
Let $\mu,w$ be positive measures on $T^{2}$ and $\delta>0$.
Assume that $w$ is of the product form~\eqref{eq:w-product}.
Then
\begin{equation}
\label{eq:cEcE:cor}
(\mathcal{E}_\delta[\mu])^{2}
\leq C_{\eqref{eq:cEcE}}
\delta \mathcal{E}[\mu] \abs{\mu}.
\end{equation}
\end{corollary}

\begin{corollary}
\label{l:l2}
Let $\nu,w$ be positive measures on $T^{2}$ and
\begin{equation}
\label{eq:large-energy-downset}
E := \Big\{  \bV^{\nu} > \frac{1}{4 C_{\eqref{eq:cEcE}}} \frac{\mathcal{E}[\nu]}{\abs{\nu}} \Big\} \subseteq T^{2}.
\end{equation}
Assume that $w$ is of the product form~\eqref{eq:w-product}.
Then
\begin{equation}\label{e:59}
\mathcal{E}_{E}[\nu]
:=
\sum_{\alpha \in E} w(\alpha) (\bI^{*}\nu(\alpha))^{2}
\geq
\frac{1}{2} \mathcal{E}[\nu].
\end{equation}
\end{corollary}
\begin{proof}
By Corollary~\ref{cor:cEcE} we have
\[
\mathcal{E}_{E}[\nu]
=
\mathcal{E}[\nu] - \mathcal{E}_{\frac{1}{4 C_{\eqref{eq:cEcE}}} \frac{\mathcal{E}[\nu]}{\abs{\nu}}}[\nu]
\geq
\mathcal{E}[\nu] - \bigl( C_{\eqref{eq:cEcE}} \frac{1}{4 C_{\eqref{eq:cEcE}}} \frac{\mathcal{E}[\nu]}{\abs{\nu}} \mathcal{E}[\nu] \abs{\nu} \bigr)^{1/2}
=
\mathcal{E}[\nu]/2,
\]
and the claim follows.
\end{proof}

\begin{theorem}
\label{thm:C-to-HC}
Let $\mu,w$ be positive measures on $T^{2}$.
Assume that $w$ is of the product form~\eqref{eq:w-product}.
Then
\[
[w,\mu]_{HC} \lesssim [w,\mu]_{C}.
\]
\end{theorem}
Theorem~\ref{thm:C-to-HC} is also contained in Theorem~\ref{thm:box=>HC}, but we give a separate short proof.
\begin{proof}
Without loss of generality $\mu\not\equiv 0$ and $[w,\mu]_{C}=1$.
Let
\begin{equation}
\label{eq:hereditary-const}
A := [w,\mu]_{HC} =
\sup_{E \subseteq T^{2}, \mu(E) \neq 0} \frac{\mathcal{E}[\mu\one_{E}]}{\mu(E)}
\end{equation}
be the hereditary Carleson constant.
Since $T^{2}$ is finite, the constant $A$ is finite, and there exists a maximizer $E$ for \eqref{eq:hereditary-const}.
Let $\nu := \mu\one_{E}$ and
\begin{equation}\label{e:57}
\mathcal{D} := \Big\{ \bV^{\nu} > \frac{1}{4 C_{\eqref{eq:cEcE}}} A \Big\}.
\end{equation}
Since $A=\frac{\mathcal{E}[\nu]}{\abs{\nu}}$, by Corollary~\ref{l:l2}, the trivial inequality $\nu\leq\mu$, and the Carleson condition \eqref{eq:Carleson} we have
\begin{equation}
\label{eq:Enu}
\frac{1}{2} \mathcal{E}[\nu]
\leq
\mathcal{E}_{\mathcal{D}}[\nu]
\leq
\mathcal{E}_{\mathcal{D}}[\mu]
\leq
\mu(\mathcal{D}).
\end{equation}
In particular, $\mu(\mathcal{D}) \neq 0$.
On the other hand, by \eqref{e:57}, the Cauchy--Schwarz inequality, \eqref{eq:Enu}, and the definition of $A$ we have
\begin{equation}
\label{muE}
\frac{A}{4 C_{\eqref{eq:cEcE}}} \mu(\mathcal{D})
\leq
\int_{\mathcal{D}} \bV^{\nu} \dif\mu
\le
\mathcal{E}[\nu]^{1/2} \mathcal{E}[\mu \one_{\mathcal{D}}]^{1/2}
\leq
2^{1/2} A^{1/2} \mu(\mathcal{D}).
\end{equation}
It follows that $A \leq 2^{5} C_{\eqref{eq:cEcE}}^{2}$.
\end{proof}
It is better to run this argument with $(4/5)^{2}$ in place of $1/4$ in \eqref{eq:large-energy-downset}.
Then it gives $5^{5}/2^{8} \leq 13$ in place of $2^{5}=32$.

\section{Hereditary Carleson implies Carleson embedding}
\label{Hered}

\begin{theorem}
\label{murho}
Let $\mu, \rho, w$ be positive measures on $T^{2}$.
Assume that $w$ is of the product form~\eqref{eq:w-product} and
\[
[w,\mu]_{HC} \leq 1,
\quad
[w,\rho]_{HC} \leq 1.
\]
Then
\[
\int \bV^{\mu} \dif\rho
\leq
2 C_{\eqref{eq:cEcE}}^{1/4} \abs{\mu}^{3/8} \abs{\rho}^{5/8}.
\]
\end{theorem}
This improves upon the estimate
\[
\int \bV^{\mu} \dif\rho
\leq
\mathcal{E} [\mu]^{1/2} \mathcal{E} [\rho]^{1/2}
\lesssim
\abs{\mu}^{1/2} \abs{\rho}^{1/2}
\]
that is immediate by Cauchy--Schwarz and the Carleson condition.
\begin{proof}
Let $\kappa > 0$ be chosen later.
By Lemma~\ref{lem:cEcE} and the Carleson condition we obtain
\[
\Bigl( \int \bV_{\kappa}^{\mu} \dif\rho \Bigr)^{3}
\leq  C_{\eqref{eq:cEcE}}
\kappa \abs{\mu} \abs{\rho}^{2}.
\]
Consider the down-set $E := \{\bV^{\mu} > \kappa\} \subset T^{2}$.
By the Cauchy--Schwarz inequality and the Carleson condition we have
\[
\int (\bV^{\mu}-\bV_{\kappa}^{\mu}) \dif\rho
=
\int_{E} w \bI^{*}\mu \bI^{*}\rho
\leq
\mathcal{E}_{E}[\mu]^{1/2} \mathcal{E}_{E}[\rho]^{1/2}
\leq
\mu(E)^{1/2} \abs{\rho}^{1/2}.
\]
Note that
\[
\kappa \mu(E)
\leq
\int_{E} \bV^{\mu} \dif\mu
\leq
\mathcal{E} [\mu]^{1/2} \mathcal{E} [\mu\one_{E}]^{1/2}
\leq
\abs{\mu}^{1/2} \mu(E)^{1/2}
\]
by the Carleson condition and \eqref{eq:hereditary-Carleson}, so $\mu(E)^{1/2} \leq \kappa^{-1} \abs{\mu}^{1/2}$.
Hence
\[
\int \bV^{\mu} \dif\rho
\leq
( C_{\eqref{eq:cEcE}} \kappa \abs{\mu} \abs{\rho}^{2} )^{1/3}
+
\kappa^{-1} \abs{\rho}^{1/2} \abs{\mu}^{1/2}.
\]
Choosing $\kappa = C_{\eqref{eq:cEcE}}^{-1/4} \abs{\mu}^{1/8} \abs{\rho}^{-1/8}$ gives the claim.
\end{proof}

\begin{remark}
A careful reading of the proof of Theorem~\ref{murho} gives the estimate
\[
\int \bV^{\mu} \dif\rho
\lesssim
\mathcal{E}[\mu]^{3/8} \mathcal{E}[\rho]^{3/8} \abs{\rho}^{1/4} [w,\mu]_{C}^{1/8} [w,\mu]_{HC}^{1/8}.
\]
\end{remark}

\begin{theorem}
\label{thm:HC-to-CE}
Let $\mu, w : T^{2} \to [0,\infty)$.
Assume that $w$ is of the product form~\eqref{eq:w-product}.
Then
\[
[w,\mu]_{CE} \lesssim [w,\mu]_{HC}.
\]
\end{theorem}

The argument below is similar to the proof of \cite[Theorem 7.1.1]{AH96}.
\begin{proof}
Without loss of generality $[w,\mu]_{HC}=1$.
Let $\psi : T^{2} \to [0,\infty)$ and consider
\[
\psi\mu = \int_{t=0}^{\infty} \mu_{t} \dif t,
\quad
\mu_{t} := \mu \one_{ \{ \psi > t \} }.
\]
Then $[w,\mu_{t}]_{HC} \leq [w,\mu]_{HC} = 1$ for every $0<t<\infty$.
Expanding the square we obtain
\begin{align*}
\int w (\bI^* (\psi\mu) )^2
&\sim
2 \int_{0<s<t<\infty} \int_{T^{2}} w (\bI^*\mu_{s}) (\bI^*\mu_{t}) \dif s \dif t
\\ \text{by Theorem~\ref{murho} } & \lesssim
2 \int_{0<s<t<\infty} \abs{\mu_{s}}^{3/8} \abs{\mu_{t}}^{5/8} \dif s \dif t
\\ & =
2 \int_{r=0}^{1} \int_{t=0}^{\infty} t \abs{\mu_{rt}}^{3/8} \abs{\mu_{t}}^{5/8} \dif t \dif r
\\ & =
2 \int_{r=0}^{1} r^{-3/4} \int_{t=0}^{\infty} ( r^{2} t \abs{\mu_{rt}})^{3/8} ( t \abs{\mu_{t}} )^{5/8} \dif t \dif r
\\ \text{by H\"older } & \leq
2 \int_{r=0}^{1} r^{-3/4} \Bigl( \int_{t=0}^{\infty} r^{2} t \abs{\mu_{rt}} \dif t \Bigr)^{3/8} \Bigl( \int_{t=0}^{\infty} t \abs{\mu_{t}} \dif t \Bigr)^{5/8} \dif r
\\ & =
\Bigl( \int_{r=0}^{1} r^{-3/4} \dif r \Bigr) \Bigl( 2 \int_{t=0}^{\infty} t \abs{\mu_{t}} \dif t \Bigr)
\\ & \lesssim
2 \int_{t=0}^{\infty} t \abs{\mu_{t}} \dif t
\\ & =
\int \psi^{2} \dif \mu.
\qedhere
\end{align*}
\end{proof}

\section{Box implies (hereditary) Carleson}
Clearly the box constant \eqref{eq:box} is smaller than the hereditary Carleson constant \eqref{eq:hereditary-Carleson}.
In this section we show a converse inequality for product weights $w$.
This counterintuitive result represents a certain combinatorial property of all planar measures.

\subsection{Balancing lemma}
We include the proof of the next result for completeness.
\begin{lemma}[{\cite[Lemma 3.1]{AHMV18b}}]
\label{lem:potential-ptw-lower-bd}
Let $\nu : T^{2} \to [0,\infty)$ be a non-negative function with
\[
\mathcal{E}[\nu]
=
\int \bV^{\nu}\dif\nu
\geq
A \abs{\nu}.
\]
Then there exists a down-set $\tilde{E} \subset T^{2}$ such that for the measure $\tilde{\nu} := \nu\one_{\tilde{E}}$ we have
\[
\bV^{\tilde{\nu}} \geq \frac{A}{3}
\quad \text{on } \tilde{E},
\]
and 
\[
\mathcal{E}[\tilde{\nu}] \geq \frac{1}{3} \mathcal{E}[\nu].
\]
\end{lemma}

\begin{proof}
Replacing $\nu$ by $3\nu/A$ we may assume $A = 3$.
Let $E_0 := T^{2}$ and $\nu_0 := \nu \one_{E_0}$.
We then define inductively
\[
E_{k+1} := E_{k} \setminus \{ \bV^{\nu_{k}} \leq 1 \},
\quad
\nu_{k+1} := \nu \one_{E_{k+1}}.
\]
The sequence $(E_{k})$ consists of down-sets in $T^{2}$ and is decreasing, and since $T^2$ is finite it must stabilize, that is, $E_{N+1}=E_{N}$ for some $N$.
By construction we have
\[
\bV^{\nu_{N}} \geq 1 \text{ on } E_{N}.
\]
Let $\sigma_{k} := \nu_{k} - \nu_{k+1} = \nu \one_{E_{k} \setminus E_{k+1}}$.
Then
\begin{equation}\notag
\begin{split}
\mathcal{E}[\nu]
&=
\int\bV^{\nu_{0}}\dif \nu_{0}
=
\int\bV^{\nu_{1}} \dif\nu_{1} + \int\bV^{\nu_{0}} \dif\sigma_{0} + \int\bV^{\sigma_{0}}\dif \nu_{1}
\leq
\int\bV^{\nu_{1}} \dif\nu_{1} + 2 \int\bV^{\nu_{0}} \dif\sigma_{0}
\\ &\leq
\mathcal{E}[\nu_{1}] + 2 \abs{\sigma_{0}}
\leq \dotsb \leq
\mathcal{E}[\nu_{N}] + 2 \abs{\sigma_{0}} + \dotsb + 2 \abs{\sigma_{N-1}}
\\ &\leq
\mathcal{E}[\nu_{N}] + 2 \abs{\nu}.
\end{split}
\end{equation}
Since $\mathcal{E}[\nu] \geq 3\abs{\nu}$ by assumption, we obtain
\[
\mathcal{E}[\nu_{N}] \geq \frac13 \mathcal{E}[\nu].
\]
This gives the conclusion with $\tilde{E} := E_{N}$.
\end{proof}

\subsection{Stopping  time argument}
Let
\[
\bV^{\mu}_{good,\epsilon}(\omega)
:=
\sum_{\substack{\alpha \geq \omega :\\ \sum_{\omega \leq \beta \leq \alpha} w(\beta) \bI^{*} \mu(\beta) > \epsilon}} w(\alpha) \bI^{*}\mu(\alpha).
\]
\begin{lemma}
\label{lem:ptwise}
Let $\mu,w$ be positive measures on $T^{2}$ and $\epsilon>0$.
Then for every $\omega \in T^{2}$ at least one of the following conditions holds.
\begin{enumerate}
\item $\bV^{\mu}_{good,\epsilon}(\omega) > \epsilon$, or
\item $\bV^{\mu}_{4 \epsilon}(\omega) \geq \bV^{\mu}(\omega)/2$.
\end{enumerate}
\end{lemma}
\begin{proof}
Fix $\omega \in T^{2}$ with
\begin{equation}
\label{eq:6}
\bV^{\mu}_{good,\epsilon}(\omega) \leq \epsilon.
\end{equation}
The set
\begin{equation}
\label{eq:7}
\mathcal{U} := \{ \alpha \geq \omega \colon \sum_{\omega \leq \beta \leq \alpha} w(\beta) \bI^{*}\mu(\beta) > \epsilon\}
\end{equation}
is an up-set.
If $\mathcal{U} = \emptyset$, then in fact $\bV^{\mu}(\omega) \leq \epsilon$, so $\bV^{\mu}_{4 \epsilon}(\omega) = \bV^{\mu}(\omega)$.
Suppose henceforth $\mathcal{U} \neq \emptyset$.
Let
\begin{equation}
\label{eq:8}
\mathcal{W} := \{ \alpha \geq \omega \colon \bV^{\mu}(\alpha) \leq 4 \epsilon\}.
\end{equation}
Without loss of generality we may assume $\omega \not\in \mathcal{W}$, since otherwise $\bV^{\mu}_{4\epsilon}(\omega) = \bV^{\mu}(\omega)$.
For $\alpha\in \mathcal{U}$ we have
\[
\bV^{\mu}(\alpha)
=
\sum_{\alpha' \geq \alpha} w(\alpha') \bI^{*}\mu(\alpha')
\leq
\sum_{\alpha' \in \mathcal{U}} w(\alpha') \bI^{*}\mu(\alpha')
=
\bV^{\mu}_{good,\epsilon}(\omega)
\leq
\epsilon
\leq
4 \epsilon,
\]
so $\mathcal{U} \subseteq \mathcal{W}$.

Let $\alpha_{0,x}$ be minimal such that there exists $\alpha_{0,y}$ with $\alpha_{0} = \alpha_{0,x}\times \alpha_{0,y} \in \mathcal{W}$, and let $\alpha_{0,y}$ be minimal with this property.
By assumption $\alpha_{0} \neq \omega$. For $j\geq 0$ construct inductively
\[
\alpha_{j+1,x} := \min \{ \tau_{x} : \tau_{x} \times \alpha_{j,y} \in \mathcal{U} \},
\quad
\alpha_{j+1,y} := \min \{ \tau_{y} : \alpha_{j+1,x} \times \tau_{y} \in \mathcal{W} \}.
\]
Notice $\alpha_{j+1,x} \geq \alpha_{j,x}$ and $\alpha_{j+1,y} \leq \alpha_{j,y}$.
We stop at $J$ for which either the minimum in the definition of $\alpha_{J+1,x}$ is taken over the empty set or $\alpha_{J,y} = \omega_{y}$.

We claim that for every $j$, except possibly $j=J$ if $J>0$, we have
\begin{equation}
\label{eq:stopping-lower-bd}
\sum_{\substack{\alpha_{j,x} \leq \tau_{x} < \alpha_{j+1,x}\\ \alpha_{j,y} \leq \tau_{y}}} (w \bI^{*} \mu)(\tau_{x}\times \tau_{y})
\geq 2 \epsilon.
\end{equation}
Here we interpret the restriction by $\alpha_{j+1,x}$ as nonexistent if $j=J$.

\begin{figure}
\centering
\begin{tikzpicture}[scale=0.5]
\draw[fill=lightgray] (1,7) -- (1,5) -- (2,5) -- (2,4) -- (3,4) -- (3,3) -- (5,3) -- (5,2) -- (6,2) -- (6,1) -- (10,1) -- (10,7) -- cycle;
\draw[fill=gray] (2,7) -- (2,6) -- (3,6) -- (3,5) -- (5,5) -- (5,4) -- (6,4) -- (6,3) -- (8,3) -- (8,2) -- (10,2) -- (10,7) -- cycle;
\draw[step=1.0,black,thin] (0,0) grid (10,7);
\draw[ultra thick] (3,2) rectangle (10,7);
\draw[pattern=dots] (0,0) rectangle (8,3);
\draw[pattern=north west lines] (3,3) rectangle (6,7);
\draw (0.5,0) node[below] {$\omega_{x}$} (3.5,0) node[below] {$\alpha_{j,x}$} (7.5,0) node[below] {$\tilde\alpha_{j,x}$};
\draw (0,0.5) node[left] {$\omega_{y}$} (0,2.5) node[left] {$\tilde\alpha_{j,y}$} (0,3.5) node[left] {$\alpha_{j,y}$};
\draw[black,thin,fill=lightgray] (11,3) rectangle (12,4) (12,3.5) node[right] {$\mathcal{W}$};
\draw[black,thin,fill=gray] (11,5) rectangle (12,6) (12,5.5) node[right] {$\mathcal{U}$};
\end{tikzpicture}
\caption[Proof of the lower bound \eqref{eq:stopping-lower-bd}]%
{Proof of the lower bound \eqref{eq:stopping-lower-bd} for the hatched area:\\
  $\tikz{ \draw[pattern=north west lines] (0,0) rectangle (0.3,0.3); } \geq
  \tikz{ \draw[ultra thick] (0,0) rectangle (0.3,0.3); } -
  \tikz{ \draw[pattern=dots] (0,0) rectangle (0.3,0.3); } -
  \tikz{ \draw[fill=gray] (0,0) rectangle (0.3,0.3); }$.}
\label{fig:stopping-lower-bound}
\end{figure}

Consider first the case that $\alpha_{j,y} > \omega_{y}$ (see Figure~\ref{fig:stopping-lower-bound}).
Let $\tilde{\alpha}_{j,y} \in T_{y}$ be the maximal element with $\alpha_{j,y} > \tilde{\alpha}_{j,y} \geq \omega_{y}$.
Let $\tilde{\alpha}_{j,x} \in T_{x}$ be the maximal element with $\tilde{\alpha}_{j,x} \geq \omega_{x}$ and $\tilde{\alpha}_{j,x} \times \tilde{\alpha}_{j,y} \not\in \mathcal{U}$.
Then
\begin{multline*}
\sum_{\substack{\alpha_{j,x} \leq \tau_{x} < \alpha_{j+1,x}\\ \alpha_{j,y} \leq \tau_{y}}} (w \bI^{*} \mu)(\tau_{x}\times \tau_{y})
\\ \geq
\bV^{\mu}(\alpha_{j,x} \times \tilde{\alpha}_{j,y}) - \sum_{\omega_{x} \leq \tau_{x} \leq \tilde{\alpha}_{j,x}} (w \bI^{*}\mu)(\tau_{x} \times \tilde{\alpha}_{j,y}) - \bV^{\mu}_{good,\epsilon}(\omega)
\geq
4 \epsilon - \epsilon - \epsilon,
\end{multline*}
where in the last inequality we used the definition~\eqref{eq:8} of $\mathcal{W}$ to bound the first term, the definition~\eqref{eq:7} of $\mathcal{U}$ to bound the middle term, and the hypothesis~\eqref{eq:6} to bound the last term, see Figure~\ref{fig:stopping-lower-bound}.

Consider now the case that $\alpha_{j,y} = \omega_{y}$.
Then necessarily $j=J$ and as we do not prove the relation for $j=J$, $J>0$ we may assume $J=0$ and so $j=0$.
By assumption there is no restriction by $\al_{1,x}$ as $j=J$.
Since $\omega \not\in \mathcal{W}$, we have $\alpha_{0,x} > \omega_{x}$.
Let $\tilde{\alpha}_{x} \in T_{x}$ be the maximal element with $\alpha_{0,x} > \tilde{\alpha}_{x} \geq \omega_{x}$.
Let $\tilde{\alpha}_{y} \in T_{y}$ be the maximal element with $\tilde{\alpha}_{y} \geq \omega_{y}$.
Then by construction $\tilde{\alpha}_{x} \times \tilde{\alpha}_{y} \not\in \mathcal{W}$, hence also $\tilde{\alpha}_{x} \times \tilde{\alpha}_{y} \not\in \mathcal{U}$.
It follows that
\begin{multline*}
\sum_{\substack{ \tau_{x} \geq \alpha_{0,x}\\  \tau_{y} \geq \alpha_{0,y}}} (w \bI^{*} \mu)(\tau_{x}\times \tau_{y})
\\ \geq
\bV^{\mu}(\tilde{\alpha}_{x} \times \alpha_{y,0}) - \sum_{\omega_{y} \leq \tau_{y} \leq \tilde{\alpha}_{y}} (w \bI^{*}\mu)(\tilde{\alpha}_{x} \times \tau_{y})
\geq
4 \epsilon - \epsilon
=
3 \epsilon,
\end{multline*}
where in the last inequality we again used the definition~\eqref{eq:8} of $\mathcal{W}$ to bound the first term and the definition~\eqref{eq:7} of $\mathcal{U}$ to bound the second term.
This finishes the proof of \eqref{eq:stopping-lower-bd}.

Since the summation sets in \eqref{eq:stopping-lower-bd} are disjoint and contained in $\mathcal{W}$ we obtain
\[
\bV^{\mu}_{4 \epsilon}(\omega)
=
\sum_{\alpha \in \mathcal{W}} w(\alpha)\bI^*\mu(\alpha)
\geq
2\epsilon \max(J,1)
\geq
\epsilon (J+1)
\]
On the other hand,
\[
\bV^{\mu}(\omega) - \bV^{\mu}_{4 \epsilon}(\omega)
=
\sum_{\alpha \geq \omega : \alpha \not\in \mathcal{W}} w(\alpha)\bI^*\mu(\alpha)
\leq
\sum_{j=0}^{J} \sum_{\substack{\alpha \geq \omega : \\ \alpha_{x} \lneq \alpha_{j,x}, \alpha_{y} \lneq \alpha_{j-1,y} }} w(\alpha)\bI^*\mu(\alpha)
\leq
\epsilon (J+1),
\]
since each inner sum is over a rectangular region and the maximum of that region is not in $\mathcal{U}$.
The last two displays show that the second alternative in the statement of the lemma holds.
\end{proof}

\subsection{Box condition implies hereditary Carleson}
\label{boxHC}
The next result contains the last missing inequality in Theorem~\ref{thm:main}.
\begin{theorem}
\label{thm:box=>HC}
Let $\nu,w : T^{2} \to [0,\infty)$ be positive functions.
Assume that $w$ is of the product form~\eqref{eq:w-product}.
Then
\[
[w,\nu]_{HC} \lesssim [w,\nu]_{Box}.
\]
\end{theorem}
\begin{proof}
By scaling we may assume $[w,\nu]_{Box}=1$ without loss of generality.
Let $A := [w,\nu]_{HC}$.
Let $E \subset T^{2}$ be a subset such that $\mu=\nu\one_{E} \neq 0$ and $\mathcal{E}[\mu] \geq \frac34 A \abs{\mu}$.
By Lemma~\ref{lem:potential-ptw-lower-bd} there exists a further subset $\tilde{E} \subset T^{2}$ such that $\tilde{\mu} := \mu \one_{\tilde{E}}$ satisfies
\[
\bV^{\tilde{\mu}} \geq \frac{A}{4} \text{ on } \tilde{E}
\]
and $\tilde\mu \neq 0$.
Thus, replacing $\mu$ by $\tilde{\mu}$, we may assume $\bV^{\mu} \geq A/4$ on $\supp \mu$.

By Corollary~\ref{cor:cEcE} we know
\[
\int \bV^{\mu}_{4\epsilon}\dif \mu
\lesssim
(\epsilon \mathcal{E}[\mu] \abs{\mu})^{1/2}
\leq
(\epsilon A)^{1/2} \abs{\mu}.
\]
Taking $\epsilon = cA$ for a small constant $c$ we can ensure that $\mu \{ \bV^{\mu}_{4\epsilon} \leq A/10 \} \geq \abs{\mu}/2$.
By Lemma~\ref{lem:ptwise} we have $\bV^{\mu}_{good,\epsilon}>\epsilon$ on that set, so
\begin{equation}
\label{eq:good-energy-dominates}
\int \bV^{\mu}_{good,\epsilon} \dif\mu
\gtrsim
\epsilon \abs{\mu}
\geq
c \mathcal{E}[\mu].
\end{equation}
We claim that for a sufficiently small absolute constant $0<\theta$ we have
\begin{equation}
\label{eq:good-energy-dominates2}
\mathcal{E}[\mu]
\lesssim
\sum_{\alpha : \theta\epsilon \bI^{*}\mu(\alpha) \leq \mathcal{E}_{\alpha}[\mu]} w(\alpha) (\bI^{*} \mu(\alpha))^{2}.
\end{equation}
Indeed, suppose that $\alpha$ is such that
\[
\theta\epsilon \bI^{*}\mu(\alpha)
>
\mathcal{E}_{\alpha}[\mu]
=
\sum_{\omega \leq \alpha} \mu(\omega) \bV^{\mu}_{\alpha}(\omega),
\quad
\bV^{\mu}_{\alpha}(\omega) = \sum_{\beta : \omega \leq \beta \leq \alpha} w(\beta) (\bI^{*}\mu)(\beta).
\]
Then by the Tchebyshov inequality we have
\[
\sum_{\omega \leq \alpha : \bV^{\mu}_{\alpha}(\omega) \leq \epsilon} \mu(\omega)
=
\bI^{*}\mu(\alpha) - \sum_{\omega \leq \alpha : \bV^{\mu}_{\alpha}(\omega) > \epsilon} \mu(\omega)
\geq (1-\theta) \bI^{*}\mu(\alpha).
\]
It follows that
\begin{multline*}
\sum_{\alpha : \theta\epsilon \bI^{*}\mu(\alpha) > \mathcal{E}_{\alpha}[\mu]} w(\alpha) (\bI^{*} \mu(\alpha))^{2}
\leq
\sum_{\alpha} w(\alpha) \bI^{*} \mu(\alpha) \frac{1}{1-\theta} \sum_{\omega \leq \alpha : \bV^{\mu}_{\alpha}(\omega) \leq \epsilon} \mu(\omega)
\\ =
\frac{1}{1-\theta} \sum_{\omega} \mu(\omega) \sum_{\alpha \geq \omega : \bV^{\mu}_{\alpha}(\omega) \leq \epsilon} w(\alpha) \bI^{*} \mu(\alpha)
=
\frac{1}{1-\theta} \sum_{\omega} \mu(\omega) (\bV^{\mu}-\bV^{\mu}_{good,\epsilon})(\omega).
\end{multline*}
Therefore
\begin{multline*}
\sum_{\alpha : \theta\epsilon \bI^{*}\mu(\alpha) \leq \mathcal{E}_{\alpha}[\mu]} w(\alpha) (\bI^{*} \mu(\alpha))^{2}
\geq
\int \bV^{\mu} \dif\mu - \frac{1}{1-\theta} \int (\bV^{\mu}-\bV^{\mu}_{good,\epsilon}) \dif\mu
\\ =
- \frac{\theta}{1-\theta} \int \bV^{\mu} \dif\mu + \frac{1}{1-\theta} \int \bV^{\mu}_{good,\epsilon} \dif\mu.
\end{multline*}
By \eqref{eq:good-energy-dominates} the latter term is $\gtrsim \mathcal{E}[\mu]$, and it dominates the first term if $\theta$ is sufficiently small.
This finishes the proof of the claim \eqref{eq:good-energy-dominates2}.

By Corollary~\ref{cor:cEcE} again and since $\bV^{\mu} \geq A/4$ on $\supp \mu$ we also have
\begin{equation}
\label{eq:small-V-energy}
\mathcal{E}_{c' A}[\mu] \lesssim (c')^{1/2} \mathcal{E}[\mu].
\end{equation}
Taking $c'$ sufficiently small and combining \eqref{eq:small-V-energy} with \eqref{eq:good-energy-dominates2} we obtain
\[
\mathcal{E}[\mu]
\lesssim
\sum_{\alpha \in \mathcal{R}} w(\alpha) (\bI^{*}\mu(\alpha))^{2},
\quad
\mathcal{R} = \{ \alpha \colon \theta\epsilon \bI^{*}\mu(\alpha) \leq \mathcal{E}_{\alpha}[\mu], \bV^{\mu}(\alpha) \geq c' A\}.
\]
For each $\alpha \in \mathcal{R}$ we have
\[
\theta\epsilon \bI^{*}\mu(\alpha)
\leq
\mathcal{E}_{\alpha}[\mu]
\leq
\mathcal{E}_{\alpha}[\nu]
\leq
[w,\nu]_{Box} \bI^{*}\nu(\alpha)
=
\bI^{*}\sigma(\alpha),
\]
where $\sigma = \nu \one_{F}$, $F = \{ \beta \in T^{2} \colon \exists \alpha\in\mathcal{R}, \alpha\geq\beta\}$.
It follows that
\begin{equation}
\label{eq:2}
\epsilon^{2} \mathcal{E}[\mu] \lesssim \mathcal{E}[\sigma].
\end{equation}
On the other hand, using definition of $A$, the fact that $\bV^{\mu} \gtrsim A$ on $\supp\sigma$, and the Cauchy--Schwarz inequality we obtain
\begin{equation}
\label{eq:3}
\mathcal{E}[\sigma]
\leq
A \abs{\sigma}
\lesssim
\int \bV^{\mu} \dif\sigma
\leq
\mathcal{E}[\mu]^{1/2} \mathcal{E}[\sigma]^{1/2}.
\end{equation}
From \eqref{eq:3} we obtain $\mathcal{E}[\sigma] \lesssim \mathcal{E}[\mu]$, and inserting this into \eqref{eq:2} gives $\epsilon \lesssim 1$.
By the choice of $\epsilon$ this in turn implies $A \lesssim 1$.
\end{proof}

\subsection{A comparison with a result of E.~Sawyer}
\label{sec:Sawyer}
The setting in Sawyer's  paper \cite{Saw} can be reduced to the following situation.
Assume that the weight $w$ has a very special structure: it is supported on ancestors of one small square.
Namely, there exists a point $\omega_0 \in \partial T^2$ such that
\begin{equation}
\label{eq:supp-w-ancestors}
\supp w \subset \Set{\alpha \in T^{2} \colon \alpha\geq \omega_0 }.
\end{equation}
If such $w$ is identically $1$ on ancestors of one small square $\omega_0$, then it is a product weight, and we know the answer: condition \eqref{e:sawcond.3} below is necessary and sufficient for embedding.
But for general weight supported on ancestors of one small square $\omega_0$, a careful reading of \cite{Saw} provides the following result.
\begin{theorem}[{cf.\ \cite[Theorem 1(A)]{Saw}}]
\label{thm:sawONE}
Suppose that the weight $w$ satisfies the support condition \eqref{eq:supp-w-ancestors} for some small square $\omega_{0} \in T^{2}$.
Then $[w,\mu]_{CE}$ is finite if and only for some $A < \infty$ and every $\beta \in T^{2}$ with $\beta \geq \omega_{0}$ the following conditions hold:
\begin{subequations}\label{e:sawcond}
\begin{align}
\label{e:sawcond.1}
\mathbb{I}^*\mu(\beta)\mathbb{I}w(\beta)
& \leq
A^2,\\
\label{e:sawcond.2}
\sum_{\alpha\geq\beta\geq\omega_0}\mu(\alpha)(\mathbb{I}w(\alpha))^2
&\leq
A^2\sum_{\alpha\geq\beta}w(\alpha),\\
\label{e:sawcond.3}
\sum_{\omega_0\leq\alpha\leq\beta}w(\alpha)(\mathbb{I}^*\mu(\alpha))^2
&\leq
A^2 \sum_{\alpha\leq \beta}\mu(\alpha).
\end{align}
\end{subequations}
No two of these conditions suffice to ensure $[w,\mu]_{CE}<\infty$.
\end{theorem}
\begin{remark}
The last condition \eqref{e:sawcond.3} is just the box condition \eqref{eq:box}.
In other words, if we restrict the weight to be supported only on the hooked rectangles, but drop the requirement  that it has a product structure, we see that the single box test \eqref{eq:box} is getting replaced by three single box tests for the pair $(\mu,w)$.

To summarize: Sawyer considers the weight concentrated on the ancestors of a single point, and we consider the weight with product structure restrictions.
It would be interesting to find a characterization of the Carleson embedding \eqref{bIstar} for general $w$, which would also cover any subgraph of the bitree.
One can even consider such a problem on an arbitrary directed graph without directed cycles, but even for the bitree the problem seems to be very difficult.

\end{remark}
\section{General setting for constructing counterexamples}
\label{general}
Let $T^2 = T^{2}_{N}$ be a finite (but very deep) dyadic bi-tree, that is, the set of all dyadic rectangles in the square $Q_0=[0,1]^2$ with side lengths at least $2^{-N}$ ordered by inclusion.
We denote the set of minimal elements of this bi-tree, that is, the small squares of size $2^{-N}\times 2^{-N}$, by $\pd T^2$, and elements of this set will be denoted by $\om$.
We denote sets of $\om$'s by $E\subset \partial T^2$ and identify them with their union, so we will write $Q \subset E$ if $Q$ is covered by the elements of $E$.

In all examples the measure $\mu$ will be supported on the square $Q_{0}$.
As in \eqref{eq:mu-on-dyadic-tree} we identify it with a function on $T^{2}$ by setting $\tilde{\mu}(\omega) := \mu(\omega)$ for $\omega \in \pd T^{2}$ and $\tilde{\mu}(Q) := 0$ for $Q \not\in \pd T^{2}$.
Then $\bI^{*}\tilde{\mu}(Q) = \mu(Q)$.
With this convention the box condition \eqref{eq:box} for the measure $\mu$ and weight $w=\{w_Q\}$ becomes
\begin{equation}
\label{e:Boxcond}
\sum_{Q\in T^2, \,Q\subset R}\mu^2(Q)w_Q \leq C \mu(R),\quad \text{for any } R\in T^2.
\end{equation}
The Carleson condition \eqref{eq:Carleson} becomes
\begin{equation}
\label{e:Carlcond}
\sum_{Q\in T^2,\,Q\subset E}\mu^2(Q)w_Q \leq C \mu(E),\quad \text{for any } E\subset (\partial T)^2,
\end{equation}
the hereditary Carleson (or Restricted Energy) condition \eqref{eq:hereditary-Carleson} becomes
\begin{equation}
\label{e:REC}
\sum_{Q\in \mathcal{D}}\mu^2(Q\cap E)w_Q \leq C\mu(E),\quad \text{for any } E\subset (\partial T)^2,
\end{equation}
and the Carleson embedding \eqref{bIstar} becomes
\begin{equation}
\label{e:imbed}
\sum_{Q\in\mathcal{D}}\left(\int_{Q}\varphi \dif\mu\right)^2w_Q \leq C\int_{Q_0}\varphi^2 \dif\mu\quad
\text{for any } \varphi\in L^2(Q_0,\dif\mu).
\end{equation}
The implications
\[
\eqref{e:Boxcond} \impliedby \eqref{e:Carlcond} \impliedby \eqref{e:REC} \impliedby \eqref{e:imbed}
\]
hold for arbitrary measures $\mu$ and weights $w$.
For product weights $w$ the converse implications hold by Theorem~\ref{thm:main}, so all these conditions are in fact equivalent.
We will show that the converse implications do not hold in general.
We will do this by constructing $N$-coarse measures $\mu$ and weights $w$ on finite bi-trees $T^2$ of depth $N$ such that the discrepancies between box, Carleson, REC, and embedding constants grow with $N$.


\section{Box condition does not imply Carleson condition}
\label{boxNotC}
In \cite{Car} Carleson constructed families $\cR$ of dyadic sub-rectangles of $Q=[0,1]^2$ having the following two properties:
\begin{equation}
\label{boxCa}
\forall R_0 \in T^2, \quad \sum_{R\subset R_0, R\in \cR} m_2(R) \le C_0 m_2(R_0)\,,
\end{equation}
but
\begin{equation}
\label{Ca}
\sum_{ R\in \cR} m_2(R) > C_1 m_2(\cup_{R\in \cR} R)\,,
\end{equation}
with arbitrarily large ratios $C_1/C_0$, where $m_2$ is the planar Lebesgue measure.
Choosing $\mu=m_{2}$ and
\[
w_R := \begin{cases} \frac{1}{m_2(R)}, & R\in \cR,
\\
0, & \text{otherwise}
\end{cases}
\]
we can identify the left-hand sides of \eqref{boxCa} and \eqref{Ca} with the left-hand sides of \eqref{e:Boxcond} and \eqref{e:Carlcond}, respectively.
Hence the box condition \eqref{e:Boxcond} holds with constant $C_{0}$, while the Carleson condition \eqref{e:Carlcond} can only hold with constant $\geq C_{1}$.

The weight $w$ is rather wild here.
But there is also a counterexample with $w_{R} \in \{0,1\}$ for all $R$, see \cite{MPV}.

\section{Carleson condition does not imply REC}
\label{sec:Car-not-REC}
Our aim here is to show that for general $w,\mu$ the Carleson condition \eqref{e:Carlcond} is no longer sufficient for the embedding \eqref{e:imbed} or even the Restricted Energy Condition \eqref{e:REC}.
Namely we prove the following statement.
\begin{proposition}\label{p:71}
For any $\delta>0$ there exists a number $N \in \mathbb{N}$, a weight $w : T^2_N \to \{0,1\}$, and a measure $\mu$ on $\partial T^2$ such that $\mu$ satisfies the Carleson condition \eqref{e:Carlcond} with the constant $C_{\mu} = \delta$:
\begin{equation}\label{e:771}
\sum_{Q\subset E}\mu^2(Q)w_Q \leq \delta\mu(E),\quad \text{for any }E\subset (\partial T)^2,
\end{equation}
but there exists a set $F \subset Q_{0}$ such that
\begin{equation}\label{e:772}
\sum_{Q\in\mathcal{D}}\mu^2(Q\cap F)w_Q > \mu(F),
\end{equation}
hence the constant in \eqref{e:REC} is at least $1$.
\end{proposition}
We intend to give two examples of this kind.
The first example is quite simple and is inspired by the counterexample for $L^2$-boundedness of the biparameter maximal function (as it should be by Proposition~\ref{prop:CE=Car*max^2}).
In this example the weight $w$ is supported on a very small subset of the bi-tree, which differs greatly from the original graph.
The second example is somewhat more involved, but the weight $w$ is supported on a much larger portion of the bi-tree; in fact it has the monotonicity property $w_R \geq w_Q$ for $R\supseteq Q$.
Nevertheless there are not enough rectangles in the support of $w$ to have the Carleson-REC equivalence.

\subsection{Notation}
In Sections~\ref{sec:Car-not-REC} and \ref{sec:REC-not-CE} we consider dyadic bi-trees $T^2 = T_N^2$ of large but finite depth $N$.
We will construct examples in which the ratio between the constants goes to $\infty$ as $N\to\infty$.

We denote by $\om_{0} := [0,2^{-N}]^2$ the left lower corner of the unit square.
Given a dyadic rectangle $R = [a,b]\times[c,d]$ let
\begin{align*}
R^{+\circ} &:= [(a+b)/2,b] \times [c,d],\\
R^{\circ+} &:= [a,b] \times [(c+d)/2,d], \\
R^{++} &:= [(a+b)/2,b] \times [(c+d)/2,d],
\end{align*}
be its right half, upper half, and upper right quadrant, respectively.

\subsection{A simple example}
\label{sec:Car-not-REC:simple}
Let $Q_{i} = [0,2^{-i+1}] \times [0,2^{-N+i}]$ for $j=i,\dotsc,N$.
Let measure $\mu$ have mass $1$ on $\om_{0}$ and each of $Q_i^{++}$, and mass $0$ everywhere else.
Let
\[
w_R :=
\begin{cases}
1 & \text{if } R \in \{ \om_{0}, Q_1, Q_2, \dotsc, Q_{N} \}, \\
0 & \text{else}.
\end{cases}
\]
So we have $N+1$ weights $w_{R}$ equal to $1$.
For the set $E=\om_{0}$ we have
\[
\cE[\mu|E] =
\mu(\omega_{0})^2 +\sum_{i=1}^N \mu(\om_{0}\cap Q_i)^2 = (N+1) = (N+1) \mu(E).
\]
So the REC constant \eqref{e:REC} is $\geq N+1$.

Denoting $Q_0:=\om_{0}$, for an arbitrary $E \subseteq \partial T^{2}$ we have
\[
\cE_E[\mu] =
\sum_{R\subset E,\, w_R\neq 0} \mu(R)^2
=
\sum_{j: Q_j\subset E} \mu(Q_j)^2.
\]
Then since $Q_{i}^{++} \cap Q_{j} = \emptyset$ unless $i\in \Set{0,j}$, we have
\[
\cE_E[\mu]
\leq \sum_{j: Q_j\subset E} 2^2
\leq 4 \mu(E)
\]
So the Carleson condition \eqref{e:Carlcond} holds with constant $4$.

\subsection{The lack of maximal principle matters}
\label{maxP}
In this section we construct a more complicated example in which the Carleson condition holds, but the restricted energy condition fails.
The weight $w$ still has values either $0$ or $1$, but the support $\mathcal{R}$ of $w$ is an \emph{up-set}, that is, it contains every ancestor of every rectangle in $\mathcal{R}$.

The example is based on the fact that potentials on bi-tree may not satisfy the maximal principle.
So we start with constructing an $N$-coarse $\mu$ such that we have
\begin{equation}
\label{le1}
\bV^\mu \lesssim 1\quad \text{on}\,\, \supp\mu,
\end{equation}
but
\begin{equation}
\label{logN}
\max \bV^\mu \ge \bV^\mu(\om_0) \gtrsim \log N.
\end{equation}

We define a collection of rectangles
\begin{equation}\label{e:773}
Q_j :=  [0,2^{-2^j}] \times  [0,2^{-2^{-j}N}] ,\quad j=1,\dotsc, M\approx \log N.
\end{equation}
Now we put 
\begin{equation}\label{e:774}
\begin{split}
\mathcal{R} &:= \{R:\; Q_j\subset R \text{ for some }j=1\dots M\}\\
w_{Q} &:= \one_{\mathcal{R}}(Q)\\
\mu(\omega) &:= \frac{1}{N}\sum_{j=1}^M \frac{1}{|Q_j^{++}|}\one_{Q_j^{++}}(\omega).
\end{split}
\end{equation}
here $|Q|$ denotes the total amount of points $\omega \in (\partial T)^2\cap Q$, i.e.\ the amount of the smallest possible rectangles (of size $2^{-N}\times 2^{-N}$) in $Q$.\par
Observe that on $Q_j$ the measure is basically a uniform distribution of the mass $\frac{1}{N}$ over the upper right quarter $Q_j^{++}$ of the rectangle $Q_j$ (and these quadrants are disjoint).

To prove \eqref{le1} we fix $\om\in Q_j^{++}$ and split
\[
\bV^\mu(\om)= \bV^\mu_{Q_j^{++}}(\om) +\mu(Q_j^{\circ+})+ \mu(Q_j^{+\circ})+ \bV^\mu(Q_j),
\]
where the first term sums up $\mu(Q)$ for $Q$ between $\om$ and $Q_j^{++}$.
It is easy to see that $\bV^\mu_{Q_j^{++}}(\om) \lesssim \frac{1}{N}$ (the left-hand side is a double geometric sum).
Trivially $\mu(Q_j^{\circ+})+ \mu(Q_j^{+\circ})\le  \frac{2}{N}$.
The non-trivial part is the estimate
\begin{equation}
\label{VQ}
\bV^\mu(Q_j) \lesssim 1\,.
\end{equation}

For each dyadic rectangle $R \supseteq \omega_{0}$ and each $j'$ we have
\begin{equation}
\label{eq:alternative}
\text{either } Q_{j'} \subseteq R, \text{ or } Q_{j'}^{++} \cap R = \emptyset.
\end{equation}
Moreover, since the sides of rectangles $Q_{j}$ are nested, the set $\{ j' : Q_{j'} \subseteq R \}$ is an interval that contains $j$.
For an interval of integers $[m,m+k]$ let
\[
C^{[m, m+k]} := \{ R \supseteq \omega_{0} : \{ j' : Q_{j'} \subseteq R \} = [m,m+k] \}.
\]
Since each rectangle in $C^{[m,m+k]}$ contains $[0,2^{-2^{m}}] \times [0,2^{-2^{-m-k}N}]$, we have
\begin{equation}
\label{eq:C-interval-upper-bd}
\# C^{[m,m+k]} \leq (2^{m}+1)(2^{-m-k}N+1) \lesssim 2^{-k} N.
\end{equation}
It follows that
\begin{equation}
\label{eq:4}
\bV^{\mu}(Q_{j})
=
\sum_{[m,m+k] \ni j} (\# C^{[m,m+k]}) (k+1) \frac{1}{N}
\lesssim
\sum_{k\geq 0} (k+1)^{2} 2^{-k} N \frac{1}{N}
\lesssim
1.
\end{equation}
This shows \eqref{VQ}, and hence \eqref{le1} is also proved.

\medskip

Now we will estimate $\bV^\mu(\om_0)$ from below.
To this end we need a more careful \emph{lower} bound on $\# C^{[m, m+k]}$.
The set $C^{\{j\}}$ contains all rectangles $R$ that contain $Q_{j}$ and are contained in $[0,2^{-2^{j-1}-1}]\times [0,2^{-2^{-j-1}N-1}]$, so
\begin{equation}
\label{eq:C-interval-lower-bd}
\# C^{\{j\}} \geq 2^{j-1} \cdot 2^{-j-1} N \gtrsim N.
\end{equation}
Hence
\begin{equation}
\label{bVmu_at_om0}
\bV^\mu(\om_0)
\ge
\sum_{j=1}^{M} (\# C^{\{j\}}) \frac{1}{N}
\gtrsim
M.
\end{equation}
This shows \eqref{logN} as $M\asymp \log N$.

\medskip

\begin{remark}
\label{Mx}
In this example $\bV^\mu\le 1$ on $\supp\mu$, and $\mathrm{cap}(\om_0) = \frac{1}{N^2}$. Denote $\la:= \log N$. Then
\begin{equation}
\label{expBELOW}
\mathrm{cap} \{\om: \bV^\mu\ge \la\} \ge  e^{-c\la},\, c>0\,.
\end{equation}
Here capacity is the bi-tree capacity defined e.g.\ in \cite{AMPS18}. So the example above essentially  uses the non-existence of maximal principle for bi-tree potential.
It would be interesting to prove the opposite inequality, namely, that
while the maximal principle fails for the bi-tree potential, the set where it fails has small capacity:
\begin{equation}
\label{expABOVE}
\mathrm{cap} \{\om: \bV^\mu\ge \la\} \le  e^{-c\la},\, c>0\,.
\end{equation}
Notice that Lemma \ref{lem:cEcE} (see also \cite[Theorem 1.6]{AMPS18}, \cite[Lemma 2.6]{AHMV18b}) gives us a much weaker estimate
\begin{equation}
\label{cubeABOVE}
\mathrm{cap} \{\om: \bV^\mu\ge \la\} \le  \frac{C}{\la^3}\,.
\end{equation}
\end{remark}

\bigskip

Now we construct the second example of $\nu$ and $w$ such that the Carleson condition holds, but the REC (restricted energy condition) fails.
The weight $w$ is chosen as in \eqref{e:774}, so this time it is the indicator function of an up-set.
With the measure $\mu$ that we have just constructed we put
\[
\nu := \mu +\nu|\om_0,
\]
where $\nu|\om_0$ is the uniformly distributed over $\om_0$ measure of total mass $\frac{1}{N}$.

\subsection{REC constant is large}

Let us first give a lower bound for the REC constant.
Consider $F=\om_{0}$.
Then by \eqref{eq:C-interval-lower-bd} we have
\[
\cE[\nu|F]
\geq
\sum_{j=1}^{M} (\# C^{\{j\}}) \nu(\omega_{0})^{2}
\gtrsim
MN \cdot \nu(\om_0)^{2}.
\]
This shows that $[w,\nu]_{HC} \gtrsim \nu(\om_0)\cdot NM = M$.

\begin{figure}
\centering
\begin{tikzpicture}[scale=0.25]
\draw[step=1.0,gray,thin] (0,0) grid (-34,-34);
\draw (-34,-34) circle (0.1) node[below left] {$0$} (-33,-34) circle (0.1) node[below] {$2^{-N}$} (-8,-34) circle (0.1) node[below] {$2^{-8}$} (-4,-34) circle (0.1) node[below] {$2^{-4}$} (0,-34) circle (0.1) node[below] {$1$};
\draw (-34,-33) circle (0.1) node[left] {$2^{-N}$} (-34,-8) circle (0.1) node[left] {$2^{-8}$} (-34,-4) circle (0.1) node[left] {$2^{-4}$} (-34,0) circle (0.1) node[left] {$1$};
\foreach \j in {0,1,2,3,4,5} {
  \draw[fill] (-2^\j,-32/2^\j) rectangle +(-1,-1);
}
\foreach \j in {0,1,2,3,4} {
  \draw[fill=darkgray] (-2*2^\j,-32/2^\j) rectangle +(-1,-1);
}
\foreach \j in {0,1,2} {
  \draw[fill=lightgray] (-8*2^\j,-32/2^\j) rectangle +(-1,-1);
}
\draw[pattern=north west lines] (-33,-33) rectangle +(-1,-1);
\end{tikzpicture}
\caption[Collections of rectangles defined in \eqref{e:773} and \eqref{lies_in_2k} in logarithmic coordinates]{
  Collections of rectangles defined in \eqref{e:773} and \eqref{lies_in_2k} in logarithmic coordinates.
  \tikz{ \draw[pattern=north west lines] (0,0) rectangle (0.25,0.25); }: $\omega_{0}$,
  \tikz{ \draw[fill] (0,0) rectangle (0.25,0.25); }: $Q_{j}^{++}=Q_{0,j}^{++}$,
  \tikz{ \draw[fill=darkgray] (0,0) rectangle (0.25,0.25); }: $Q_{1,j}^{++}$,
  \tikz{ \draw[fill=lightgray] (0,0) rectangle (0.25,0.25); }: $Q_{2,j}^{++}$.
}
\end{figure}

\subsection{Carleson constant is small}
Next we will verify that the Carleson condition \eqref{e:Carlcond} holds with a small constant.
We may remove from the sum on the left-hand side of \eqref{e:Carlcond} all rectangles $Q \not\in \mathcal{R}$.
Then we can replace $E$ by the union of remaining $Q$'s without changing the left-hand side and decreasing the right-hand side.
Hence we may reduce to the case when $E$ is a union of members of $\mathcal{R}$.
By \eqref{eq:alternative} it follows that for each $j$ we have either $Q_{j} \subseteq E$ or $Q_{j}^{++} \cap E = \emptyset$.
Let $\mathcal{J} := \{ j : Q_{j} \subseteq E\}$.
Then we obtain
\begin{equation}
\label{eq:5}
LHS\eqref{e:Carlcond}
\leq
\sum_{[m,m+k] \subseteq \mathcal{J}} \sum_{Q \in C^{[m,m+k]}} \bigl( (k+1)/N + 1/N \bigr)^{2}.
\end{equation}
Using \eqref{eq:C-interval-upper-bd} this implies
\begin{align*}
LHS\eqref{e:Carlcond}
&\lesssim
\sum_{[m,m+k] \subseteq \mathcal{J}} 2^{-k} N (k+2)^{2} (1/N)^{2}
\\ &\lesssim
(\# \mathcal{J})/N
\\ &\leq
\mu(E)
\\ &\leq
\nu(E),
\end{align*}
so that $[w,\nu]_{Car} \lesssim 1$.

\begin{remark}
\label{rem:sum-of-tensor-products}
Let $w = \sum_{j=1}^{M} w_{j}$ be the sum of $M$ weights each of which is of the tensor product form~\eqref{eq:w-product}.
By subadditivity of the Hereditary Carleson constant in the first argument, Theorem~\ref{thm:main}, and monotonicity of the Carleson constant in the first argument, we have
\begin{equation}
\label{eq:C-to-HC-sum-of-tensors}
[w,\nu]_{HC}
\leq
\sum_{j=1}^{M} [w_{j},\nu]_{HC}
\lesssim
\sum_{j=1}^{M} [w_{j},\nu]_{C}
\leq
\sum_{j=1}^{M} [w,\nu]_{C}
=
M [w,\nu]_{C}.
\end{equation}
A modification of the above example shows that the constant $M$ in this inequality is optimal.
Namely, let
\begin{equation}
\label{e:774again}
\begin{split}
&\mathcal{R}_j := \{R:\; Q_j\subset R\}, \,\,j=1\dots M,
\\
&w_{Q} := \sum_{j=1}^M\one_{\mathcal{R}_j}(Q)\,.
\end{split}
\end{equation}
With this new $w$ the above proof of the bound $[w,\nu]_{Car} \lesssim 1$ remains valid.
Indeed, the only change is that additional factors $(k+1)$ appear in \eqref{eq:4} and \eqref{eq:5}.
The proof of the lower bound $[w,\nu]_{HC} \gtrsim M$ remains unchanged.

The example in Section~\ref{sec:Car-not-REC:simple} also shows that the growth rate $M$ in \eqref{eq:C-to-HC-sum-of-tensors} cannot be improved.
\end{remark}

\section{REC does not imply embedding}
\label{sec:REC-not-CE}

In this section we emulate the previous construction, we start with $\{Q_j\}$ and measure $\mu$ but instead of adding $\om_0$ we will add a more sophisticated piece of measure.

We define $Q_{j},\mu,\mathcal{R},w$ as in the previous section.
We continue with denoting 
\[
Q_{0, j}:= Q_j,\quad \mu_0:=\mu\,\, \text{from the previous section}\,.
\]
Next we continue with defining a sequence of collections $\mathcal{Q}_k,\; k=0,\dotsc,K \approx \log M$ of dyadic rectangles as follows
\begin{equation}
\label{lies_in_2k}
\mathcal{Q}_k := \left\{ Q_{k,j} := \bigcap_{i=j}^{j+2^k-1}Q_{0,i},\; j=1,\dotsc,M-2^k \right\},\; k=1,\dotsc,K.
\end{equation}
In other words, $\mathcal{Q}_{k}$ consists of the intersections of $2^k$ consecutive elements of the basic collection $\mathcal{Q}_0$.
The total amount of rectangles in $\mathcal{Q}_k$ is denoted by $M_k = M-2^{k}+1$.

For $k=1,\dotsc,K$ let
\[
\mu_k(\omega) := \frac{2^{-2k}}{N}\sum_{j=1}^{M_k} \frac{1}{|Q_{k,j}^{++}|}\one_{Q_{k,j}^{++}}(\omega),\quad \omega \in (\partial T)^2,
\]
and define
\[
\mu := \mu_0 + \sum_{k=1}^{K}\mu_k.
\]
\subsection{Embedding constant is large}
\label{MI}
By duality the inequality \eqref{bIstar} is equivalent to the Carleson embedding inequality
\begin{equation}
\label{dual}
\int (\bI (fw))^2 \dif\mu \leq [w,\mu]_{CE} \sum f^2\cdot w.
\end{equation}
We test the inequality \eqref{dual} with the function
\[
f(R) := \mu_{0}(R) = \bI^{*}\mu_{0}(R).
\]
Using \eqref{le1} we obtain
\begin{equation}
\label{RHS}
\sum f^2\cdot w
=
\int \bV^{\mu_0} \dif\mu_0
\lesssim
\norm{\mu_{0}}
=
\frac{M}{N}\,.
\end{equation}
On the other hand, by definition \eqref{lies_in_2k} and replacing $M$ by $2^{k}$ in \eqref{bVmu_at_om0} we obtain
\begin{equation}
\label{bVmu0_on_muk}
\bV^{\mu_0}(Q_{k, j}) \gtrsim 2^k N \cdot \frac{1}{N} = 2^k.
\end{equation}
It follows that
\begin{equation}
\label{LHS}
\int (\bI (fw))^2 \dif\mu
=
\int (\bV^{\mu_0})^2 \dif\mu
=
\sum_{k=1}^K \int (\bV^{\mu_0})^2 \dif\mu_k
\gtrsim
\sum_{k=1}^K  2^{2k} \|\mu_k\|
\sim
\frac{M}{N} \log M.
\end{equation}
Substituting \eqref{RHS} and \eqref{LHS} in \eqref{dual} we obtain $[w,\mu]_{CE} \gtrsim \log M$.

\subsection{REC constant is small}
\label{RECsmall}
We claim that $[w,\mu]_{HC} \lesssim 1$.
This means that for any collection $\cA$ of dyadic rectangles, setting $A := \cup_{R\in \cA} R$, we have
\begin{equation}
\label{smREC}
\cE[\mu|A]\lesssim \mu(A).
\end{equation}
To show \eqref{smREC} let $\nu_k:=\mu_k|A$, $k=0,\dots,K$.
Then
\[
\cE[\mu|A]
=
\sum_{n,k} \int \bV^{\nu_{n}} \nu_{k}
\leq
2 \sum_{n\geq k} \int \bV^{\nu_{n}} \nu_{k}
\leq
2 \sum_{n\geq k} \int \bV^{\mu_{n}} \nu_{k}.
\]
Since $\supp \nu_{k} \subseteq \supp \mu_{k}$ it suffices to show
\begin{equation}
\label{eq:1}
\sum_{n\geq k} \bV^{\mu_{n}} \lesssim 1
\quad\text{on}\quad \supp \mu_{k}.
\end{equation}
The claim \eqref{eq:1} has the advantage that it does not depend on $\cA$ any more.

For every $R\in\cR$ we have
\begin{multline*}
\mu_{n}(R)
=
2^{-2n} \# \Set{ Q_{n,j} \subseteq R }
\leq
2^{-2n} (\# \Set{ Q_{0,j} \subseteq R } + 2^{n} )\\
\leq
2^{-n} (\# \Set{ Q_{0,j} \subseteq R } + 1 )
\leq
2 \cdot 2^{-n} \mu_{0}(R).
\end{multline*}
It follows that
\[
\bV^{\mu_n}(Q_{k, j})
\lesssim
2^{-n} \bV^{\mu_0}(Q_{k, j})
\leq
2^{-n} \sum_{i=j}^{j+2^{k}-1} \bV^{\mu_{0}}(Q_{0,i})
\lesssim
2^{k-n},
\]
where the last inequality follows from \eqref{le1}.
This implies \eqref{eq:1} and therefore \eqref{smREC}.

\begin{remark}
\label{logMdiscrepancy}
We can repeat the change of weight in the manner we did in Remark~\ref{rem:sum-of-tensor-products} in order to have
\[
w=\sum_{i=1}^M \tau_i\otimes \eta_i\,.
\]
The discrepancy between REC constant and embedding constant is at least of the order $\log M$.
\end{remark}
\begin{remark}
Observe also that in the examples of Sections \ref{sec:Car-not-REC} and \ref{sec:REC-not-CE} the weight $w$ was supported only on the predecessors of one fixed point on the boundary of the bi-tree. This means that the weighted bi-tree is essentially a weighted subset of $\mathbb{Z}^2$, which  puts us  in the context of \cite{Saw}.
As we mentioned before, see Theorem~\ref{thm:sawONE}, Sawyer showed that in this case the embedding is equivalent to \textit{three single box conditions}, in particular any combination of two of them does not imply the embedding anymore. Here we supplemented his results by showing that even the dual multiple box test   is not enough to get any of the relations in Theorem \ref{thm:main}. In other words, the analog of Chang--Fefferman type condition (we called it Carleson condition)  and even stronger REC (= hereditary Carleson condition) condition is not enough to ensure embedding even in the simplest case  when weight $w$ was supported only on the predecessors of one fixed point on the boundary of the bi-tree. 
\end{remark}

\appendix
\section{Maximal function versus embedding}
\label{sec:Chang}
The definition of the Hardy operator~\eqref{eq:Hardy} and its adjoint~\eqref{eq:adj-Hardy} makes sense on an arbitrary finite partially ordered set $\cP$.
Given a weight $\mu$ on $\cP$ we define the corresponding maximal operator by
\begin{equation}
\label{eq:max-op}
\mathcal{M}_{\mu} \psi(\omega) := \sup_{\alpha \geq \omega} \La\abs{\psi}\Ra (\alpha),
\quad
\La\psi\Ra (\alpha) := \frac{\bI^{*}(\psi \mu)(\alpha)}{\bI^{*}(\mu)(\alpha)},
\end{equation}
with the convention $0/0=0$.
This definition recovers the usual dyadic maximal operator on the tree and the bi-parameter maximal operator on the bi-tree.
We can also define Carleson \eqref{eq:Carleson} and Carleson embedding \eqref{bIstar} constants for pairs of weights on an arbitrary partially ordered set.
A classical argument \cite[p.\ 236]{St} gives the following relation between these constants.
\begin{proposition}
\label{prop:CE=Car*max^2}
Let $\mu$ be a measure on a partially ordered set $\cP$.
Then
\begin{equation}
\label{eq:CE=Car*max^2}
\sup_{w : [w,\mu]_{C} \leq 1} [w,\mu]_{CE} = \norm{\mathcal{M}_{\mu}}_{L^{2}(\mu)\to L^{2}(\mu)}^{2}.
\end{equation}
\end{proposition}

\begin{example}
If $\cP = T$ is a usual tree, then the maximal function $\cM_{\mu}$ is essentially the martingale maximal function, and it is well-known that it is bounded on $L^{p}(\mu)$ with norm at most $p'$.
In particular the right-hand side of \eqref{eq:CE=Car*max^2} equals $4$.
This is the sharp constant in the Carleson embedding theorem on the tree \cite{NTV99}.
\end{example}

\begin{example}
If $\cP=T^{2} = T_{x} \times T_{y}$ is a bi-tree and $\mu = \mu_{x} \times \mu_{y}$ is a product measure, then the two-parameter maximal operator \eqref{eq:max-op} can be majorized by the composition of two one-parameter maximal operators
\[
\mathcal{M}_{\mu} \psi \le \mathcal{M}_{x,\mu_{x}} \mathcal{M}_{y,\mu_{y}} \psi,
\]
which are also defined by \eqref{eq:max-op} but on the trees $T_{x}$ and $T_{y}$.
Using $L^{2}$ bounds for the one-parameter maximal operators we see that the right-hand side of \eqref{eq:CE=Car*max^2} is bounded by $16$.
Hence for product measures $\mu$ and arbitrary weights $w$ Proposition~\ref{prop:CE=Car*max^2} gives the implication \eqref{Carpara1} $\implies$ \eqref{para2}.
\end{example}

\begin{proof}[Proof of Proposition~\ref{prop:CE=Car*max^2}]
We begin with the inequality $\leq$ in \eqref{eq:CE=Car*max^2}.
Let $\psi : \cP \to [0,\infty)$ be a non-negative function.
Then
\begin{align*}
&\sum_{\alpha} w(\alpha) \bI^{*}(\psi\mu)(\alpha)^{2}
=
\sum_{\alpha} w(\alpha) \bI^{*}(\mu)(\alpha)^{2} \int_{0}^{\La \psi \Ra (\alpha)} 2s \dif s
\\ & =
\int_{0}^{\infty} 2s \sum_{\alpha : \La \psi \Ra (\alpha) > s} w(\alpha) \bI^{*}(\mu)(\alpha)^{2} \dif s
\\&\leq
\int_{0}^{\infty} 2s \sum_{\alpha : \mathcal{M}_{\mu}\psi(\alpha) > s} w(\alpha) \bI^{*}(\mu)(\alpha)^{2} \dif s
\leq
[w,\mu]_{C} \int_{0}^{\infty} 2s \mu\{\alpha : \mathcal{M}_{\mu}\psi(\alpha) > s\}  \dif s
\\ &=
[w,\mu]_{C} \norm{\mathcal{M}_{\mu}(\psi)}_{L^{2}(\mu)}^{2}
\leq
[w,\mu]_{C} \norm{\mathcal{M}_{\mu}}_{L^{2}(\mu)\to L^{2}(\mu)}^{2} \norm{\psi}_{L^{2}(\mu)}^{2}.
\end{align*}
Here we have used that the superlevel sets $\Set{\alpha : \mathcal{M}_{\mu}(\psi)(\alpha) > s}$ are down-sets.
Taking the supremum over $\psi$ and $w$ we obtain the inequality $\leq$ in \eqref{eq:CE=Car*max^2}.

Now we will show the inequality $\geq$ in \eqref{eq:CE=Car*max^2}.
Let $\psi : \cP \to [0,\infty)$ be a non-negative function such that $\norm{\psi}_{L^{2}(\mu)} > 0$.
For each $\alpha \in \cP$ with $\bI^{*}(\psi\mu)(\alpha) \neq 0$ let
\[
A'(\alpha) := \Set{ \omega \leq \alpha : \mathcal{M}_{\mu}\psi(\omega) = \La \psi \Ra (\alpha)},
\]
and let $A'(\alpha):=\emptyset$ otherwise.
Enumerate $\cP = \Set{\alpha_{1},\alpha_{2},\dotsc}$ and set
\[
A(\alpha_{j}) := A'(\alpha_{j}) \setminus \bigcup_{j'<j} A'(\alpha_{j'}).
\]
Then
\begin{align*}
\sum_{\beta \in \cP} (\mathcal{M}_{\mu}\psi)^{2}(\beta) \mu(\beta)
&=
\sum_{\alpha} \sum_{\beta \in A(\alpha)} \La \psi \Ra (\alpha)^{2} \mu(\beta)
\\ &=
\sum_{\alpha} w(\alpha) (\bI^{*}(\psi\mu)(\alpha))^{2},
\end{align*}
where
\[
w(\alpha) := (\bI^{*}(\mu)(\alpha))^{-2} \sum_{\beta \in A(\alpha)} \mu(\beta)
\]
with the convention $w(\alpha)=0$ if $A(\alpha)=\emptyset$ (in which case we might be dividing by zero in the above formula).
Since the sets $A(\alpha)$ are disjoint and consist of elements smaller than $\alpha$, for every down-set $\mathcal{D} \subseteq \cP$ we have
\[
\sum_{\alpha \in \mathcal{D}} w(\alpha) (\bI^{*}(\mu)(\alpha))^{2}
=
\sum_{\alpha \in \mathcal{D}} \sum_{\beta \in A(\alpha)} \mu(\beta)
\leq
\sum_{\alpha \in \mathcal{D}} \mu(\alpha),
\]
so $[w,\mu]_{C} \leq 1$.
On the other hand, by the above calculation
\[
[w,\mu]_{CE} \geq \frac{ \norm{\mathcal{M}_{\mu}\psi}_{L^{2}(\mu)}^{2} }{ \norm{\psi}_{L^{2}(\mu)}^{2} }.
\]
By first taking the supremum over $w$ with $[w,\mu]_{C} \leq 1$ and then over $\psi$ we obtain the inequality $\geq$ in \eqref{eq:CE=Car*max^2}.
\end{proof}

\subsection{A sparse proof}
In this section we give an alternative proof of a special case of the inequality $\leq$ in Proposition~\ref{prop:CE=Car*max^2} using an argument going back to \cite{Verb} and using \cite{Dor}.

\begin{proof}[Alternative proof of $\leq$ in Proposition~\ref{prop:CE=Car*max^2}]
Let $X$ be a standard Borel space with an atom-free measure $\mu$.
Suppose that $\cP$ is a finite collection of Borel subsets of $X$ ordered by inclusion, which we denote by $\leq$, $\geq$.
Assume that each $Q\in\cP$ is the disjoint union of minimal elements of $\cP$ contained in it (as is the case in the dyadic bi-tree of finite depth).
Let $\tilde\mu(\omega) := \mu(\omega)$ for $\omega$ minimal and $0$ otherwise. We prove the inequality for $\mu$ as we identify $\mu$ with $\tilde\mu$.
Let $w$ be a sequence with $[w,\tilde\mu]_{C} = 1$. This implies for all $\Omega$ which are union of elements of $\cP$
\begin{equation}
\label{alCarl1}
\sum_{Q \subseteq \Omega} w(Q) \mu(Q)^2 \le \mu(\Omega).
\end{equation}

as $\tilde\mu(\Om)=\mu(\Om)$ for such $\Om$.

By \cite[Theorem 1.3]{TH} (see also \cite{B} for a geometric proof) this implies the existence of pairwise disjoint measurable subsets $E_Q\subset Q$, $Q\in \cP$, such that
\begin{equation}
\label{choice}
w(Q)\cdot \mu(Q)^2 \leq \mu(E_Q).
\end{equation}
Let an arbitrary positive function $\psi$ on $\cP$. We define a new function $\tilde\psi$ on $ \bigcup\limits_{\om\in \cP}\om$ by $\tilde\psi(x):=\psi(\om(x))$ where $\om(x)$ is the unique minimal element of $\cP$ containing $x$. 

We can easily deduce that $\cM_{\tilde\mu}\psi(\om(x))=\cM'_{\mu}\tilde\psi(x)$ for $\cM'_{\mu}\tilde\psi(x):=\sup\limits_{\substack{Q\in\cP\\Q\ni x}} \frac{1}{\mu(Q)}\int_Q \tilde\psi\dif \mu$. From this, we also get $\norm{\cM_{\tilde\mu}\psi}_{L^{2}(\cP,\tilde\mu)}=\norm{\cM'_{\mu}\tilde\psi}_{L^{2}(\bigcup\limits_{\om\in \cP}\om, \mu)}$. Finally, recall that $\bI^*\tilde\mu(Q)=\mu(Q)$. Then,

\begin{align*}
&\sum_{Q\in \cP} \bigr( \bI^{*}(\psi\tilde\mu)(Q) \bigr)^2 w(Q)
=
\sum_{Q\in \cP} \langle \psi\rangle_{\tilde\mu}(Q)^2 w(Q)\cdot \mu(Q)^2
\leq
\sum_{Q\in \cP} \langle \psi\rangle_{\tilde\mu}(Q)^2 \cdot \mu(E_Q)
\\ &\leq
\sum_{Q\in \cP} \Big(\inf_{\om\leq Q}\cM_{\tilde\mu}\psi(\om) \Big)^2\cdot \mu(E_Q)
\leq
\sum_{Q\in \cP} \int_{E_Q}\Big(\cM'_{\mu}\tilde\psi \Big)^2 \dif\mu
\\ &\leq
\int\limits_{\bigcup\limits_{\om\in \cP}\om} \Big(\cM'_{\mu}\tilde\psi \Big)^2 \dif\mu
=
\int_{\cP} \Big(\cM_{\tilde\mu}\psi \Big)^2 \dif\tilde\mu
\leq 
\norm{\cM_{\tilde\mu}}_{L^{2}(\tilde\mu) \to L^{2}(\tilde\mu)}^{2} \norm{\psi}_{L^{2}(\tilde\mu)}^{2}.
\end{align*}

Taking the supremum over $\psi$ we obtain $[w,\tilde\mu]_{CE} \leq \norm{\cM_{\tilde\mu}}_{L^{2}(\tilde\mu) \to L^{2}(\tilde\mu)}^{2}$.
\end{proof}



\end{document}